\documentclass[a4paper, 11pt, twoside, english]{article}

\usepackage[utf8]{inputenc}
\usepackage[T1]{fontenc}
\usepackage{pifont}
\usepackage{times}   

\usepackage[english]{babel}
\usepackage{geometry}
\usepackage{amsmath,amsfonts,amssymb,amsthm}
\usepackage[dvipsnames]{xcolor}

\definecolor{skobeloff}{rgb}{0.0, 0.48, 0.45}
\definecolor{amber}{rgb}{1.0, 0.75, 0.0}
\definecolor{aogreen}{rgb}{0.0, 0.5, 0.0}
\definecolor{darksienna}{rgb}{0.24, 0.08, 0.08}
\definecolor{seagreen}{rgb}{0.18, 0.55, 0.34}
\definecolor{mahogany}{rgb}{0.75, 0.25, 0.0}

\usepackage[final,colorlinks]{hyperref}
\usepackage{graphicx}
\usepackage{cleveref}
\hypersetup{
     colorlinks=true,
     linkcolor=black,
     filecolor=blue,
     citecolor=skobeloff,
     urlcolor=skobeloff,
     }
\usepackage{mathtools}
\usepackage{pdflscape}
\usepackage{etoolbox}
\usepackage{fancyhdr}
\usepackage{lastpage}
\usepackage{ifdraft}
\usepackage[autostyle,italian=guillemets]{csquotes}
\usepackage[backend=bibtex, maxnames=50, firstinits=false, sorting=nyt]{biblatex}
\usepackage{colortbl}
\usepackage{booktabs}
\usepackage{hyphenat}
\usepackage{siunitx}
\usepackage[mathlines,pagewise]{lineno}

\graphicspath{{./images/}}

\usepackage{colortbl}
\usepackage{amssymb}
\usepackage[shortlabels]{enumitem}
\setlist[enumerate,1]{label=\textnormal{(\emph{\roman*})}}
\usepackage{cases}

\allowdisplaybreaks

\newcommand{\dd}{\mathop{}\!\mathrm{d}}

\usepackage{array,multirow}

\usepackage{tikz}
\usetikzlibrary{matrix,arrows.meta}

\usepackage{changes} 
\setdeletedmarkup{\color{red}\sout{#1}}

\DeclareMathOperator*{\e}{e}

\DeclareMathOperator*{\dist}{dist}

\DeclareMathOperator*{\diag}{diag}

\DeclareMathOperator*{\Range}{Range}
\DeclareMathOperator*{\blkdiag}{blkdiag}

\DeclareMathOperator*{\Lip}{Lip}
\newcommand{\R}{\mathbb{R}}

\newcommand{\C}{\mathbb{C}}

\usepackage{bbm}

\theoremstyle{plain}
\newtheorem{theorem}{Theorem}[section]
\newtheorem{lemma}[theorem]{Lemma}
\newtheorem{proposition}[theorem]{Proposition}

\newtheorem{definition}[theorem]{Definition}

\theoremstyle{definition}
\newtheorem{remark}[theorem]{Remark}
\newtheorem{assumption}{Assumption}

\numberwithin{equation}{section}
\numberwithin{table}{section}
\numberwithin{figure}{section}

\usepackage[shortlabels]{enumitem}
\setlist[enumerate,1]{label=\textnormal{(\emph{\roman*})}}

\fancypagestyle{mypagestyle}{%
  \fancyhf{}%
  \fancyhead[LO,RE]{\scriptsize On the numerical computation of $R_0$ in periodic
environments}%
  \fancyhead[RO,LE]{\scriptsize D. Breda, S. De Reggi, J. Ripoll}%
  \fancyfoot[LE,RO]{\scriptsize \thepage\,/\,\pageref*{LastPage}}%
}
\usepackage{authblk}
\usepackage{titling}

\usepackage{footnote}
\thanksmarkseries{arabic}

\title{On the numerical computation of $R_0$ in periodic
environments}

\author{Dimitri Breda\thanks{CDLAb - Computational Dynamics Laboratory, University of Udine, Italy}\thanksgap{2mm}$^,$\thanks{Department of Mathematics, Computer Science and Physics, University of Udine, Italy}\thanksgap{2mm}$^,$\thanksmark{5}
,\ \ Simone De Reggi\thanksmark{1}\thanksgap{2mm}$^,$\thanks{Department of Mathematics, University of Trento, Italy}\thanksgap{2mm}$^,$\thanksmark{5} ,\ \ Jordi Ripoll \thanks{Department of Computer Science, Applied Mathematics and Statistics, University of Girona, Spain}\thanksgap{2mm}$^,$\thanks{e-mail: \url{dimitri.breda@uniud.it},  \url{simone.dereggi@unitn.it}, \url{jordi.ripoll@udg.edu}}}

\date{\vspace{-7ex}}

\begin{document}
\maketitle
\begin{abstract}
\noindent We propose a novel approach to approximate the basic reproduction number $R_0$ as spectral radius of the Next-Generation Operator in time-periodic population models by characterizing the latter via evolution semigroups. Once birth/infection and transition operators are identified,
we discretize them via either Fourier or Chebyshev collocation methods. Then $R_0$ is obtained by solving a generalized matrix eigenvalue problem. The order of convergence of the approximating reproduction numbers to the true one is shown to depend on the regularity of the model coefficients, and spectral accuracy is proved. We validate the theoretical results by discussing applications to epidemiology, viz. a large-size multi-group epidemic model with periodic contact rates, and a vector-borne disease model with seasonal vector recruitment. We illustrate how the method facilitates implementation compared to existing approaches and how it can be easily adapted to also compute type-reproduction numbers.

\medskip
\noindent\textbf{Keywords:} Basic reproduction number, next-generation operator,  seasonality, evolution semigroups,  spectral approximation, epidemic/extinction threshold.

\smallskip
\noindent\textbf{MSC codes:} 34L16, 37N25, 41A10, 47D06,  47B07, 65L60, 65L70, 92D25.
\end{abstract}

\section{Introduction}
The importance of seasonality in epidemiological and ecological models has been highlighted by many authors, e.g., \cite{altizer2006seasonality, buonomo2018seasonality, grassly2006seasonal, keeling2011} and references therein. In the modeling of infectious diseases, factors such as climate, school calendars, and human behavior introduce periodic fluctuations in transmission rates, contact patterns, and population demographics \cite[chapter 5]{keeling2011}. Similarly, in ecological models, birth and death rates, migration patterns, and resource availability often follow seasonal cycles. These temporal patterns can impact crucial parameters such as the Basic Reproduction Number (BRN) $R_0$,  
which represents the average number of new infections (offsprings) generated by a typical primary case (individual) during its entire period of infectiousness (fertility) in a fully susceptible population. As such, $R_0$ is a threshold parameter that determines whether a disease/population will persist or be eliminated. In the sequel, we only refer to ``birth'' as a new infection is interpreted as a birth in the demographic sense \cite[chapter 7, p.~162]{diekmann2013mathematical}. 

Mathematical models incorporating seasonality are described by nonautonomous systems with time-periodic coefficients. 
It is well known that for autonomous structured population models $R_0$ can be computed as the spectral radius of a Next-Generation Operator (NGO) \cite{ diekmann1990}, which can be obtained by splitting the linearization around a given equilibrium into birth and transition \cite{barril2018practical}. In particular, for Ordinary Differential Equations (ODEs) the NGO reduces to a Next-Generation Matrix \cite[section 7.2]{diekmann2013mathematical}.
However, in the time-periodic case the classical definition of the BRN for structured populations must be generalized \cite[section 7.9]{diekmann2013mathematical}.
In \cite{bacaer2012, Bacaer2006,  inaba2012new, inaba2019, Wang2008}, the authors showed that in periodic environments $R_0$ can be defined as the spectral radius of a NGO acting on a space of periodic functions, extending the definition and generational interpretation of \cite{diekmann1990} for autonomous models \cite{inaba2012new, inaba2019}.\footnote{Besides \cite{Bacaer2006} and related works, many other authors have attempted to extend the definition of $R_0$ to time-periodic models. See, e.g.,  \cite{heesterbeek1995threshold2, heesterbeek1995threshold} or \cite{bacaer2012, Inaba2012b} and the references therein.} 
As in the autonomous case, the NGO can be obtained by splitting the linearization of the model around a given equilibrium or periodic solution into birth and transition \cite{inaba2019, Wang2008}.
Notably, $R_0$ is a threshold parameter that does not depend on time;  yet the problem of its computation becomes infinite-dimensional, even for ODE models (although, for some very special cases, e.g., time-periodic triangular matrices as in \cite{Mitchell2017, Wang2008}, one gets the exact $R_0$ as the one of the averaged problem).
One should be aware that autonomous models with time-averaged coefficients can overestimate or underestimate the true value of $R_0$, see \cite[chapter 7, p.~194]{diekmann2013mathematical} or \cite{Mitchell2017, Wang2008}. Therefore, conclusions drawn from this type of systems should be handled carefully.

Foundational works 
\cite{Bacaer2007,bacaer2012, Bacaer2006} 
established several methods for computing such $R_0$, namely, discretization of the NGO, Floquet multipliers for the linear periodic system, and perturbation expansions about the averaged problem, i.e., $R_0= \rho_0 + \rho_1 \epsilon + \rho_2 \epsilon^2 + \dots$, see \cite[section 3.3]{Bacaer2007}.
In this paper, we introduce a novel numerical method for computing $R_0$ in periodic environments based on the theory of evolution semigroups \cite{chicone1999, inaba2019, Thieme2009}.\footnote{Even though we focus on time-periodic models, we mention that the extension of the definition and generational interpretation of $R_0$ to general time-heterogeneous problems has been discussed in \cite{inaba2012new, inaba2019}, to which we refer the reader for further details. See also \cite{Thieme2009}.} Our approach is easier to implement: indeed, although we deal with an infinite-dimensional generalized eigenvalue problem, the latter is built directly from the nonautonomous population model avoiding the auxiliary constructions of monodromy matrices or the solution of initial value problems. The relevant time-dependent operators (i.e., birth and transition) are discretized via either Fourier or Chebyshev pseudospectral collocation methods: the former uses trigonometric polynomials, the latter algebraic ones. The resulting approach is conceptually straightforward and computationally efficient, not requiring additional zero-finding methods or other complementary numerical tools.

Given the above advantages, we focus the treatment on systems of ODEs: beyond being widely used in epidemiological and ecological modeling, this choice allows us to better highlight the novel contribution coming from the interpretation of the NGO via evolution semigroups. Extending the proposed approach to continuous age/size, spatial, or stage structures is plausible, but certainly introduces additional complexities as relevant models become systems of Partial Differential Equations 
or Delay Equations, 
and thus we reserve to pursue this objective in future work.

The proposed Fourier and Chebyshev schemes require different proofs of convergence that we rigorously provide.
For the Fourier method, the convergence is investigated as in \cite{de2024convergence} by using the well-established spectral approximation theory of \cite{Chatelin1981}. In particular, the norm convergence of the discretized operators is obtained by using interpolation error bounds in the $L^1$-norm \cite[p.~221, Theorem 3.2.9]{Mastroianni}. 
For the Chebyshev method, the convergence is proved by adapting the results of \cite{BredaKuniyaRipollVermiglio2020}. Here, we resort to interpolation error bounds in weighted spaces \cite[p.~271, eq.~4.33 and p.~274, Theorem 4.3.2]{Mastroianni}.
For both methods, we prove that the convergence order of the approximating reproduction number depends on the regularity of the model coefficients, and a spectrally accurate behavior is obtained \cite{Trefethen2000}.
We illustrate the effectiveness of our approach through two case studies: a large-scale multi-group epidemic model with periodic contact rates, and a vector-borne disease model with seasonal recruitment of susceptible vectors. In both examples, we demonstrate that our method yields accurate estimates of $R_0$, validating the results of the theoretical convergence analysis. Moreover, with the second model we also illustrate how the method can be adapted to compute Type-Reproduction Numbers (TRNs) \cite{bacaer2012model, HEESTERBEEK20073,  Inaba2012,  Heesterbeek2003}.

The paper is organized as follows. In \cref{R0perdef}, we introduce a prototype linear population model with time-periodic coefficients. With reference to this model, we illustrate the definition of $R_0$ in the periodic case, and we illustrate the state-of-the-art on the numerical methods for its computation. In \cref{Evolution}, we briefly recall some basic concepts and results on evolution semigroups. In \cref{R0andEvolution}, we discuss the connection of $R_0$ with the theory of evolution semigroups. The section is concluded with the compactness proof for the NGO. In \cref{method}, we present the numerical methods alongside details about their implementation. The convergence proofs are presented in \cref{convanal}. Finally, in \cref{results}, we present numerical results and we discuss the extension of the method to the computation of the TRN.
MATLAB demos are available on GitLab via \url{http://cdlab.uniud.it/software}.

We close this introduction with the following useful definition, where  for a linear operator $A\colon X\to Y$ and Banach spaces $X,Y$, $\|A\|_{Y\leftarrow X}\coloneqq \sup_{\|x\|_{X}=1}\|Ax\|_{Y}$.
\begin{definition}[\cite{aulbach1996, chicone1999}]\label{def:evosys}
A family $\mathcal V\coloneqq \{V(t,s)\ :\ t,s\in \R, \  t\ge s\}$ of bounded linear operators $V(t,s):Z\to Z$ on a Banach space $Z$ is called an \emph{evolutionary system} (or \emph{evolution family}) if
\begin{enumerate}[(i)]
\item $V(s, s)=I_{Z}$ for every $s\in \R$, 
\item $V(t,r)V(r, s)=V(t, s)$ for every $t,r,s\in \R$,\ $t\ge r\ge s$.
\end{enumerate}
$\mathcal V$ is called \emph{continuous} if, for every $z\in Z$, $(t, s)\mapsto V(t, s)z$ is continuous. $\mathcal V$ is called \emph{exponentially bounded} if there exists $C\ge 1$ and $\omega\in \R$ s.t. $\|V(t, s)\|_{Z \leftarrow Z}\le C{\e}^{\omega (t-s)}$, $t,s\in\R$, $t\ge s$,
in which case
\begin{align}
\label{omegaV}\omega(\mathcal V):=\inf \big\{\tilde\omega\in\R\ :\ \exists C\ge 1\text{ s.t. }
\|V(t, s)\|_{Z\leftarrow Z}\le C{\e}^{\tilde \omega (t-s)},\ t,s\in\R,\ t\ge s\big\}
\end{align}
is called \emph{exponential growth bound}.
\end{definition}

\section{Prototype model, $R_0$ and state of the art} \label{R0perdef}
We consider a prototype linear(ized) system of ODEs describing the dynamics of a population evolving in a periodic environment according to linear birth and transition processes (e.g., death, recovery, change of status). 
With reference to this model and following \cite{Bacaer2006, inaba2019, Thieme2009, Wang2008}  and \cite[section 7.9]{diekmann2013mathematical} we recall the definition of the BRN in the time-periodic setting. 

Let $\tau>0$ denote the time-period. For $u(t)\in \mathbb{R}^d$,  $t\ge0$ and $d$ a positive integer, we consider the linear Initial Value Problem (IVP)
\begin{equation}\label{ODEBM}
\left\{\setlength\arraycolsep{0.1em}\begin{array}{rlll} 
u'(t)&=& B(t) u(t) - M(t)u(t), &\quad t \ge s, \\ [0mm]
u(s)&=&u_0 \in \R^d, &\quad s\in\mathbb{R},
\end{array} 
\right.
\end{equation}
where $B, M\colon \R \to \R^{d\times d}$ are $\tau$-periodic functions that account for birth and transition, respectively. When $B\equiv0$, we expect the population to go extinct. Consequently, consider also the \textit{unperturbed} IVP
\begin{equation}\label{ODEM}
\left\{\setlength\arraycolsep{0.1em}\begin{array}{rlll} 
v'(t)&=& - M(t)v(t), &\quad t \ge s, \\ [0mm]
v(s)&=&v_0 \in \R^d, &\quad s\in\mathbb{R},
\end{array} 
\right.
\end{equation}
and let $V(t,s)$ denote the associated Principal Matrix Solution (PMS) at $s$, so that $v(t)=V(t,s)v_{0}$.
In the sequel, for $Z$ either $\R^d$ or $\R^{d\times d}$ and $1\le p\le +\infty$, let
\begin{equation*}
    L_\tau^p(\R, Z)\coloneqq \{\phi\colon \R\to Z\ :\ \phi(t)=\phi(t+\tau)\text{ for a.a. }t\in \R\text{ and } \phi_{\restriction_{[0,\tau]}}\in L^p([0, \tau], Z)\}
\end{equation*}
be equipped with $\|\phi\|_{L_\tau^p(\R, Z)}\coloneqq \|\phi_{\restriction_{[0, \tau]}}\|_{L^p([0, \tau], Z)}$.\footnote{In the sequel we will use the subscript $\tau$ to denote spaces of $\tau$-periodic functions, e.g., $\Lip_{\tau}^s$ for Lipschitz continuous functions with $s$-th Lipschitz continuous derivative, $s\ge0$.} 
We assume the following \cite{inaba2019, Thieme2009}.
\begin{assumption}\label{assBM}
\hspace{0mm}
\begin{enumerate}
\item $B\in L^\infty_\tau(\R,\R^{d\times d})$ is nonnegative;
\item $M\in L^\infty_\tau(\R,\R^{d\times d})$ is s.t. $\mathcal V\coloneqq \{V(t,s)\ :\ t,s\in \R,\  t\ge s\}$ is a continuous evolutionary system of nonnegative
operators on $\R^d$ with $\omega(\mathcal V)<0$.\footnote{The nonnegativity of $V(t,s)$
and $\omega(\mathcal V)<0$ are  ensured, e.g., if $-M(t)$ is nonzero, Metzler (i.e., all of its off-diagonal entries are nonnegative) and with nonpositive entries on the diagonal for a.a. $t\in \R$ \cite{Inaba2012b}. Note that in the theory of positive operators on real ordered Banach ($L^p$) spaces $X$, 
a nonzero bounded linear operator $A\colon X\to X$ is called \emph{positive} if it maps the positive cone $X_+:=\{\psi\in X\ :\ \psi \ge 0\}$ into itself \cite[p. 225]{Schaefer}.  In finite dimension, instead, a matrix is called \emph{positive} if all of its entries are positive and \emph{nonnegative} if all of its entries are nonnegative. For a comprehensive discussion see \cite{Inaba2012b}.}
\end{enumerate}
\end{assumption}
The $\tau$-periodicity of $M$ implies
\begin{equation}\label{periodicV}
V(t+\tau,s+\tau)=V(t,s),\quad  t,s\in \R,\quad t\ge s,
\end{equation}
and thus $\omega(\mathcal V)<0$ iff $\rho\left(V(t+\tau,t)\right)<1$ holds for a.a. $t\in \R$ \cite[Proposition 5.6]{Thieme2009}, where $\rho(A)$ denotes the spectral radius of a matrix (or in general a linear operator) $A$.

The NGO for \eqref{ODEBM} is the positive operator $\mathcal K\colon L^1_\tau(\R, \R^d)\to L^1_\tau(\R, \R^d)$ defined as 
\begin{equation} \label{KBC}
(\mathcal K\psi)(t)\coloneqq B(t)\int_{0}^\infty V(t, t-s)\psi(t-s) \dd s,\qquad t\in [0, \tau],
\end{equation}
and, following \cite{diekmann1990}, the BRN is characterized as its  spectral radius:
\begin{equation}\label{R0per}
R_0\coloneqq \rho(\mathcal K).
\end{equation}
In \eqref{KBC}, $\psi$ represents the density
of the distribution w.r.t. the time-at-birth of the current generation of individuals \cite[section 7.9]{diekmann2013mathematical}.
In \cref{Evolution}, we prove that the operator $\mathcal K$ in \eqref{KBC} is compact. Hence, if $R_0>0$, then the Krein--Rutman Theorem \cite{krein1948} ensures that $R_0$ is a dominant real eigenvalue with an associated nonnegative eigenfunction. Note also that from \eqref{KBC}-\eqref{R0per} one can characterize $R_0$ as the unique positive solution of the nonlinear scalar equation 
\begin{equation}\label{nonlineareq}
\rho\left( U^{(\lambda)}(\tau, 0) \right)=1,
\end{equation}
where $U^{(\lambda)}(\tau, 0) \in \R^{d\times d}$ is the monodromy operator associated with the linear ODE
\begin{equation*} 
u'(t)= \left(\frac{1}{\lambda}B(t)-M(t)\right) u(t),
\end{equation*}
i.e., the PMS at $0$ after one period. We refer to \cite{Bacaer2007, Wang2008},  \cite[section 7.2]{Bacaer2021b} and \cite[section 8]{diekmann2013mathematical} for further details.

Observe that, in the periodic case, the computation of $R_0$ leads to deal with either the solution of the infinite-dimensional eigenvalue problem 
\begin{equation}\label{ngoeig}
\mathcal K\psi =\lambda \psi,\qquad (\lambda, \psi)\in \R\times L^1_\tau(\R, \R^d),
\end{equation} 
for positive $\lambda$, or with the solution of the nonlinear equation \eqref{nonlineareq}. The solution of both is in general  not attainable analytically. Methods based on the discretization of \eqref{ngoeig} have been proposed in \cite{Bacaer2007,posny2014}, while methods based on the numerical solution of \eqref{nonlineareq} through the combination of ODE solvers and dichotomy methods have been proposed in \cite{Bacaer2007,Wang2008} (see \cite{Mitchell2017} for a review of all these approaches).
As for the methods based on the discretization of \eqref{ngoeig}, one of the main drawbacks is that one explicitly needs to derive the NGO \eqref{KBC}, thus compute $V(t,s)$ and approximate the integral over the halfline.
Methods based on the numerical solution of \eqref{nonlineareq} are typically more flexible, as they do not require to explicitly derive the NGO. Yet, they require to numerically compute from scratch at each iteration an approximation of the monodromy operator $U^{(\lambda)}(\tau, 0)$ in \eqref{nonlineareq}, and then to solve a nonlinear equation, which in turn requires to choose an appropriate starting point for the iteration. For all these reasons, we propose in \cref{R0andEvolution} an alternative description of the NGO, which is more convenient and effective in view of numerically approximating $R_0$. The reformulation is based on exploiting the theory of evolution semigroups, the basics of which are illustrated in the forthcoming \cref{Evolution} (first recall \Cref{def:evosys}).

\section{Evolution semigroups}
\label{Evolution}
We introduce the concept of evolution semigroup associated with an evolutionary system \cite{aulbach1996, chicone1999, inaba2019, Thieme2009, vanminh1994}. We first recall basic definitions and results on semigroups of linear operators \cite{EngelNagel2000}, and then we adapt them to the time-periodic case \eqref{periodicV}.
\begin{definition}
A family $\mathcal T\coloneqq \{T(s)\ :\ s\ge 0\}$ of bounded linear operators $T(s)\colon X\to X$ on a Banach space $X$ is called a \emph{semigroup} if:
\begin{enumerate}[(i)]
\item $T(0)=I_{X}$, 
\item $T(r)T(s)=T(r+s)$ for every $r,s\in\R$, $r,s \ge 0$.
\end{enumerate}
$\mathcal T$ is called \emph{strongly continuous} or $C_{0}$ if  $\lim_{s\to 0^+}T(s)\phi=\phi$ for every $\phi\in X$. 
\end{definition}
\begin{theorem}
Let $\mathcal T\coloneqq \{T(s)\}_{s\ge 0}$ be a $C_0$-semigroup on $X$. Then   
there exists $C\ge 1$ and $\omega\in \R$ s.t. $\|T(s)\|_{X \leftarrow X}\le C{\e}^{\omega s}$, $s\ge 0$.
\end{theorem}
\begin{definition}
For $\mathcal T\coloneqq \{T(s)\}_{s\ge 0}$ a $C_0$-semigroup on $X$
\begin{align*}
\label{omegaT}\omega(\mathcal T)\coloneqq \inf \big\{\tilde\omega\in \R \ :\ &\exists C\ge 1\text{ s.t. }
\|T(s)\|_{X\leftarrow X}\le C{\e}^{\tilde \omega s},\ s\in\R,\ s\ge 0\big\}
\end{align*}
is called \emph{exponential growth bound}.
\end{definition}
\begin{definition}
Let $\mathcal T=\{T(s)\}_{s\ge 0}$ be a $C_0$-semigroup on  $X$. The linear operator $\mathcal A\colon D(\mathcal A)\subset X\to X$ defined as
\begin{equation*}
\mathcal A \phi\coloneqq \lim_{s\to 0^+}\cfrac{T(s)\phi-\phi}{s},\qquad \phi\in D(\mathcal A)\coloneqq \left\{\phi \in X\ :\ \lim_{s\to 0^+}\cfrac{T(s)\phi-\phi}{s}\text{ exists}\right\}, 
\end{equation*}
is called the \emph{infinitesimal generator} of $\mathcal T$.
\end{definition}

We now consider a special semigroup associated to an evolutionary system $\mathcal V\coloneqq \{V(t,s)\ :\ t,s\in \R,\  t\ge s\}$ on a Banach space $Z$ along with its main properties.
\begin{definition}
The semigroup $\mathcal T\coloneqq \left\{T(s)\ :\ s\ge 0\right\}$ on $X\coloneqq L^1(\R, Z)$ given by $T(s):X\to X$ with $\left(T(s)\phi\right)(t)\coloneqq V(t, t-s)\phi(t-s),\ t\in \R$, $s\ge 0$, is called the \emph{evolution semigroup} of $\mathcal V$.
\end{definition}
\begin{theorem}[{\cite[Proposition 1]{aulbach1996}, \cite[Lemma 4]{inaba2019}, \cite[Lemma 5.1]{Thieme2009}}]\label{samebound}
Let $\mathcal T$ be the evolution semigroup on $X\coloneqq L^1(\R, Z)$ associated to an evolutionary system $\mathcal V$ on a Banach space $Z$. If $\mathcal V$ is continuous and exponentially bounded, then $\mathcal T$ is $C_0$ and $\|T(s) \|_{X\leftarrow X}=\sup_{t\in \R}\|V(t+s,t)\|_{Z\leftarrow Z}$. Moreover, $s(\mathcal A)=\omega(\mathcal T)=\omega(\mathcal V)$ holds for $\mathcal A$ the infinitesimal generator of $\mathcal T$, where $s(\mathcal A)\coloneqq \sup_{\lambda \in \sigma(\mathcal A)} \Re(\lambda)$.
\end{theorem}
In the time-periodic case \eqref{periodicV}, Theorem \ref{samebound} holds almost unchanged for $X\coloneqq L^1_\tau(\R, Z)$.
\begin{theorem}[{\cite[Proposition 5.5, Lemma 5.8 and Proposition A.2]{Thieme2009}}]\label{sameboundper}
Let $\mathcal T$ be the evolution semigroup on $X\coloneqq L_{\tau}^1(\R, Z)$ associated to a $\tau$-periodic evolutionary system $\mathcal V$ on a Banach space $Z$. If $\mathcal V$ is continuous, then it is exponentially bounded, $\mathcal T$ is $C_0$ and $\|T(s) \|_{X\leftarrow X}=\sup_{t\in \R}\|V(t+s,t)\|_{Z\leftarrow Z}$. Moreover, $s(\mathcal A)=\omega(\mathcal T)=\omega(\mathcal V)$ holds for $\mathcal A$ the infinitesimal generator of $\mathcal T$.
\end{theorem}
\begin{remark}
Note that for general $C_0$-semigroups $\mathcal T$ and their infinitesimal generators $\mathcal A$ only $s(\mathcal A)\le \omega (\mathcal T)$ holds and the inequality may be strict \cite[section 4.2]{EngelNagel2000}. However, for evolution semigroups the equality always holds thanks to the spectral mapping theorem \cite[p.~65, Theorem 3.13]{chicone1999}.
\end{remark}

\section{$R_0$ from evolution semigroups}
\label{R0andEvolution}
Following \cite{Inaba2012b, inaba2019, Thieme2009} we now illustrate the connection of the theory presented in \cref{R0perdef} with that of evolution semigroups. Let us consider the $\tau$-periodic evolutionary system $\mathcal V$ corresponding to \eqref{ODEM} with \eqref{periodicV}. Under \Cref{assBM}, \Cref{sameboundper} ensures that the associated evolution semigroup $\mathcal T$ on $X\coloneqq L^1_\tau(\R, \R^d)$ is a $C_0$-semigroup of positive operators. Let $Y\coloneqq AC_\tau(\R, \R^d)$ be the space of $\tau$-periodic absolutely continuous functions endowed with the norm $\|\phi\|_Y\coloneqq \|\phi\|_X+\|\phi'\|_X$ and let 
$\mathcal M\colon Y \to  X$ 
be defined as
\begin{equation}\label{Mper}
\left(\mathcal M \phi\right)(t)\coloneqq \phi'(t)+M(t)\phi(t), \quad t\in \mathbb{\R}.
\end{equation}
One can show that $-\mathcal M $ is the infinitesimal generator of $\mathcal T$, see \cite[section 2.2.1]{chicone1999}.
In particular, under \Cref{assBM}, \Cref{sameboundper} ensures that $s(-\mathcal M)=\omega(\mathcal V)<0$. Consequently, $\mathcal M$ is invertible with bounded inverse, and from standard results in semigroup theory\footnote{For a $C_0$-semigroup $\mathcal T:=\{T(s)\}_{s\ge 0}$ with generator $\mathcal A$ on a Banach space $X$ satisfying $\omega(\mathcal T)=s(\mathcal A)$, one has
\begin{equation*}
(\lambda-\mathcal A)^{-1}\psi=\int_0^\infty {\e}^{-\lambda s} T(s)\psi\dd s,\quad \psi\in X,
\end{equation*}
for every $\lambda\in \C$ with $\Re(\lambda)> s(\mathcal A)$.} \cite[section 2.1, Theorem 1.10]{EngelNagel2000}, we get \cite{Thieme2009} 
\begin{equation*}
 (\mathcal M^{-1}\psi)(t) = \int_0^\infty \left(T(s)\psi\right)(t) \dd s=\int_0^\infty V(t, t-s) \psi(t-s) \dd s,\quad \psi\in X.
\end{equation*}
Hence, by considering the positive linear bounded operator $\mathcal B\colon X\to X$ defined as
\begin{equation}\label{Bper}
\left(\mathcal B \phi\right)(t)\coloneqq B(t)\phi(t), \quad t\in \mathbb{\R},
\end{equation}
the NGO in \eqref{KBC} can be equivalently obtained as \cite{Inaba2012b, inaba2019, Thieme2009}
\begin{equation}\label{KBM-1}
\mathcal K=\mathcal B\mathcal M^{-1}. 
\end{equation}
The following result of compactness is then an immediate consequence.
\begin{theorem}\label{HZcompact}
The NGO $\mathcal K$ in \eqref{KBM-1} is compact, $0\in \sigma(\mathcal K)$ and $\sigma(\mathcal K)\setminus \{0\}$ is a bounded set that consists of eigenvalues only with finite algebraic multiplicity, with no accumulation points except possibly for 0. In addition, if $R_0$ in \eqref{R0per} is positive, then it is a dominant real eigenvalue with an associated real nonnegative eigenfunction.
\end{theorem}
\begin{proof}
Since $\mathcal B$ in \eqref{Bper} is bounded, it is enough to show that $\mathcal M^{-1}$ is compact. Observe that $\Range(\mathcal M^{-1})=Y$ and that $\mathcal M^{-1}\colon X\to Y$ is bounded. Indeed, by letting $\phi=\mathcal{M}^{-1}\psi$ in \eqref{Mper} for $\psi\in X$ one first gets that $(\mathcal{M}^{-1}\psi)'=\psi-M\mathcal M^{-1}\psi$ and then, under the assumptions on $M$,
\begin{align*}
\|\mathcal M^{-1}\|_{Y\leftarrow X}&=\sup_{\substack{\psi\in X,\\ \|\psi\|_X =1}} \left(\|\mathcal M^{-1}\psi\|_X+\|(\mathcal M^{-1}\psi)'\|_X\right)\\
&\le \sup_{\substack{\psi\in X,\\ \|\psi\|_X=1}} \left(\|\mathcal M^{-1}\|_{X\leftarrow X}\|\psi\|_X+\|\psi-M\mathcal M^{-1}\psi\|_X\right)\\
&\le \sup_{\substack{\psi\in X,\\ \|\psi\|_X=1}} \left(1+\|\mathcal M^{-1}\|_{X\leftarrow X} +\|M\|_{L^\infty_\tau(\R, \R^{d\times d})} \|\mathcal M^{-1}\|_{X\leftarrow X}\right)\|\psi\|_X\\
&= 1+\left(1 +\|M\|_{L^\infty_\tau(\R, \R^{d\times d})} \right)\|\mathcal M^{-1}\|_{X\leftarrow X}.
\end{align*}
The claim follows from the compactness of the immersion of $Y$ into $X$ \cite[p.~285, Theorem 9.16]{Brezis2011}, from standard results on compact operators \cite[section 3.1 and 3.2, Theorem 3.11]{Kress1989} and from the Krein--Rutmann Theorem \cite{krein1948}, see also \cite[p.~315]{Schaefer}.
\end{proof}
Let us note that, in light of \eqref{ngoeig} and \eqref{KBM-1}, $R_0$ can be equivalently computed by solving either the eigenvalue problem
\begin{equation}\label{ngoeig1}
\mathcal B\mathcal M^{-1}\psi=\lambda\psi,\qquad (\lambda,\psi )\in \mathbb \R\times X,
\end{equation}
or by solving the generalized eigenvalue problem
\begin{equation}\label{gep}
\mathcal B\phi=\lambda\mathcal M\phi,\qquad (\lambda,\phi )\in \R\times Y,
\end{equation}
the equivalence relying on $\phi=\mathcal M^{-1}\psi$. For $\lambda=R_{0}$ it is not difficult to recover
\begin{equation*}
R_{0}=  \left(\displaystyle \sum_{i,j=1}^d \int_{0}^{\tau} B_{ij}(t) \phi_j(t)\dd t\right)\left(\displaystyle\sum_{i,j=1}^d \int_{0}^{\tau} M_{ij}(t) \phi_j(t)\dd t\right)^{-1},
\end{equation*}
for $\phi\in Y_+$ the nonnegative generalized eigenfunction.
The 
characterization of $\mathcal K$ in \eqref{KBM-1} is particularly convenient for investigating methods to numerically approximate $R_0$. In fact, following the idea of \cite{BredaDeReggiScarabelVermiglioWu2022, BredaFlorianRipollVermiglio2021, BredaKuniyaRipollVermiglio2020, de2024approximating, de2024convergence, Theta2019, Kuniya2017} and as illustrated in the forthcoming section \ref{method}, we can directly discretize the operators $\mathcal B$ and $\mathcal M$ to derive a finite-dimensional version of \eqref{gep}. In this way, as anticipated at the end of section \ref{R0perdef}, we avoid to explicitly compute $\mathcal K$ through \eqref{KBC} or the monodromy operator $U^{(\lambda)}(\tau,0)$ in \eqref{nonlineareq}, thus without computing $\mathcal{M}^{-1}$, solving \eqref{ODEM} to get $V(t,s)$ or approximating integrals on the halfline. Finally, note that, in favor of numerical discretization, \eqref{gep} is posed on a space of functions $Y$ which is more regular than $X$ on which instead both \eqref{ngoeig} and \eqref{ngoeig1} are posed.

\section{The numerical approach}\label{method}
Let $N$ be a positive integer and consider a mesh
\begin{equation}\label{nodes}
\Theta_N\coloneqq \{0\le t_1<\dots <t_N\le \tau\}
\end{equation}
of $N$ distinct nodes over the period $[0,\tau]$. The underlying idea to discretize \eqref{gep} is that of approximating a function $\phi \in Y$ through a vector $\Phi \in Y_{N}\coloneqq\C^{dN}$ according to $\phi(t_i)\approx \Phi_i \in \C^d$, $i=1,\ldots,N$ (complexification is everywhere assumed as dealing with eigenvalues). This leads to the discrete (i.e., finite-dimensional) generalized eigenvalue problem
\begin{equation}\label{fingep}
\mathcal B_N \Phi=\lambda_N\mathcal M_N \Phi,\qquad (\lambda_N, \Phi) \in \C\times Y_{N}.
\end{equation}
Then $R_0$ is approximated through $$R_{0, N}\coloneqq \max\{|\lambda|\ :\  \lambda\in\C\text{ and } \mathcal 
 B_N\Phi= \lambda \mathcal  M_N \Phi\text{ for some } \Phi \in Y_{N}\}.$$
The structure of the resulting matrices $\mathcal B_N,\mathcal M_N\colon Y_{N}\to Y_{N}$ depends on the particular choice of the discretization technique.

In this paper we adopt collocation, and two alternate schemes are proposed, namely Fourier and Chebyshev collocation. On the one hand, Fourier collocation relies on using trigonometric polynomials as approximating functions \cite{Boyd2001, gottlieb2001, Trefethen2000}, so that periodicity is naturally ensured. Yet, in the presence of discontinuities in the model coefficients or in their derivatives, a piecewise extension is not immediate. On the other hand, using piecewise algebraic polynomials solves the issue, but periodicity needs to be enforced. As in this case we use a mesh of Chebyshev points, we refer to this approach as to Chebyshev collocation \cite{Boyd2001, gottlieb2001, Trefethen2000}.
To avoid an unnecessary heavy notation, we use the same symbols for both collocation schemes, which of the two we are referring to will be clear from the context. Moreover, as the piecewise extension of algebraic polynomials is a standard argument, we limit ourselves to describe the use of a single polynomial over the period.
To rigorously introduce both techniques, we assume the following.
\begin{assumption}\label{Assumption CBM}
$B, M\in C_\tau(\R,  \R^{d\times d})$. 
\end{assumption}

\subsection{Discretization via Fourier collocation}\label{sec:fourier}
We assume the following.
\begin{assumption}\label{AssFourier}
$\Theta_N$ in \eqref{nodes} is the mesh of equidistant nodes $$t_i\coloneqq \frac{\tau(i-1)}{N}, \qquad i=1,\ldots,N.$$
\end{assumption}
Let then $\mathbb F_{N}$ denote the related space of $\C^d$-valued trigonometric polynomials on $\R$ of degree at most $\lfloor N/2\rfloor$.
Let $\phi_{N}\in \mathbb F_{N}$ collocate \eqref{gep} as
\begin{equation}\label{gepFourier}
(\mathcal B\phi_{N})(t_i)=\lambda_{N}(\mathcal M\phi_{N})(t_i),\quad i=1,\ldots,N.
\end{equation}
By using the Lagrange representation $\phi_{N}(t)\coloneqq \sum_{j=1}^{N} \ell_j(t)\Phi_j$ for $\Phi_j\coloneqq\phi_{N}(t_{j})$, the cardinal properties of the Lagrange trigonometric basis $\{\ell_j\}_{j=1}^{N}$ relevant to $\Theta_N$ \cite{baltensperger2002, berrut2007} and the linearity of $\mathcal B$ and $\mathcal M$,  \eqref{gepFourier} can be rewritten as
\begin{equation}\label{gepFourier3}
B(t_i)\Phi_i =\lambda_{N}\left (\sum_{j=1}^{N} \ell_j'(t_i) \Phi_j+ M(t_i)\Phi_i\right), \quad i=1,\dots, N.
\end{equation}
By introducing the Fourier differentiation matrix $D_{N}\coloneqq [\ell_j'(t_i)]_{i,j=1,\dots, N}$ \cite{baltensperger2002, gottlieb2001}, \eqref{gepFourier3} reduces to \eqref{fingep} for
\begin{equation}\label{BN}
\mathcal B_{N}\coloneqq\blkdiag(B(t_1),\dots, B(t_{N}))
\end{equation}
and $\mathcal M_{N}\coloneqq(D_{N}\otimes I_{\C^d})+\blkdiag(M(t_1),\dots, M(t_N))$.
Above we use the MATLAB-like notation $\blkdiag$ to indicate that a matrix is block diagonal of the relevant entries. Finally, explicit entries of $D_{N}$ can be found in \cite{baltensperger2002, gottlieb2001, weideman2000}.

\subsection{Discretization via Chebyshev collocation}\label{sec:chebyshev}
We assume the following.
\begin{assumption}\label{AssCheb}
$\Theta_{N}$ in \eqref{nodes} is the mesh of Chebyshev extremal nodes $$t_i\coloneqq \frac{\tau}{2}\left[1-\cos\left(\frac{i\pi}{N}\right)\right],\qquad i=1,\ldots, N.$$ Let moreover $\Theta_{N,0}\coloneqq \{t_{0}\coloneqq0\}\cup\Theta_{N}$.
\end{assumption}
Let then $\mathbb P_{N}$ denote the related space of $\C^d$-valued algebraic polynomials on $\R$ of degree at most $N$.
Let $\phi_{N}\in \mathbb P_{N}$ collocate \eqref{gep} as
\begin{align}
 \label{gepChebBC}\phi_N(0)&=\phi_N(\tau),\\[0mm]
 \label{gepCheb}(\mathcal B\phi_N)(t_i)&=\lambda_{N}(\mathcal M\phi_N)(t_i),\quad i=1,\dots,N.
\end{align}
By using the Lagrange representation $\phi_{N}(t)\coloneqq \sum_{j=0}^{N} \ell_{j}(t)\Phi_{j}$ for $\Phi_j\coloneqq\phi_{N}(t_{j})$, the cardinal properties of the Lagrange algebraic polynomial basis $\{\ell_j\}_{j=0}^{N}$ relevant to $\Theta_{N,0}$ and the linearity of $\mathcal B$ and $\mathcal M$, \eqref{gepCheb} can be rewritten as
\begin{align}\label{gepCheb2}
B(t_i)\Phi_i &=\lambda_{N}\left (\sum_{j=1}^{N-1} \ell_j'(t_i) \Phi_j +(\ell_0'(t_i)+\ell_{N}'(t_i))\Phi_N+ M(t_i)\Phi_i\right),\quad i=1,\dots, N,
\end{align}
where \eqref{gepChebBC} has been explicitly incorporated by using $\Phi_{0}=\Phi_{N}$.
By introducing the Chebyshev differentiation matrix $D_N\coloneqq [\ell_{j}'(t_{i})]_{i,j=0,\dots, N}$ \cite{gottlieb2001, Trefethen2000}, \eqref{gepCheb2} reduces to \eqref{fingep} for $\mathcal B_N$ as in \eqref{BN} and $\mathcal M_N\coloneqq\left(\widehat {D_N} \otimes I_{\C^d}\right)+\blkdiag(M(t_1),\dots, M(t_{N}))$
with $\widehat {D_N}$ defined as
\begin{equation*}
\left[\widehat {D_N}\right]_{i,j}\coloneqq \begin{cases}\left[D_N\right]_{i,j}&\text{if }j\ne N,\\[0mm]
\left[D_N\right]_{i, 0}+\left[D_N\right]_{i,N} &\text{if }j=N,
\end{cases}
\end{equation*}
for $i,j=1,\dots, N$. Finally, explicit entries of $D_{N}$ can be found in \cite{gottlieb2001, Trefethen2000, weideman2000}.\footnote{In the MATLAB codes available at \url{http://cdlab.uniud.it/software}, the Fourier nodes and differentiation matrices are computed with the publicly available code \textsf{fourdif.m} from \cite{weideman2000}, while the Chebyshev ones are computed with \textsf{cheb.m} from  \cite{Trefethen2000}.}

\section{Convergence analysis}\label{convanal}
We now investigate the convergence of the methods presented in \Cref{method} giving error bounds under each of the following requirements.
\begin{assumption}\label{ass3}
\hspace{1mm}
\begin{enumerate}
\item \label{ass31} $B, M \in \Lip^s_\tau(\R,\R^{d\times d})$  for some integer $s\ge 0$,
\item \label{ass32} $B, M \in C_\tau^\infty(\R,\R^{d\times d})$,
\item \label{ass33} $B, M:\R\to\R^{d\times d}$ are $\tau$-periodic and  real analytic.
\end{enumerate} 
\end{assumption}
We follow two different strategies depending on whether the finite-dimensional approximation space is included or not in the original space $Y$. In the affirmative case, in fact, it is possible to invert $\mathcal M_N$ for $N$ large enough \cite{de2024convergence}. In particular, for the Fourier method $\mathbb F_N\subset Y$ holds and hence we can define an operator $\mathcal K_{N}\coloneqq \mathcal B_N\mathcal M_N^{-1}$ and, following \cite{de2024convergence}, investigate the convergence of its eigenvalues to those of $\mathcal K$ via the spectral approximation theory of \cite{Chatelin1981}. As far as the Chebyshev method is concerned, instead, $\mathbb P_N\not\subset Y$. Thus we alternatively follow the arguments of \cite[section 4]{BredaKuniyaRipollVermiglio2020}, deriving characteristic equations for both \eqref{gep} and \eqref{fingep}. We then prove the well-posedness of the collocation problem underlying \eqref{fingep} and investigate the convergence of the eigenvalues by comparing the characteristic equations via Rouch\'e's Theorem.

\subsection{Convergence of Fourier collocation}\label{subconvFourier}
Let us first prove that $\mathcal M_N$ is invertible for $N$ large enough. To this aim we introduce the restriction operator $R_{N}\colon Y\to Y_{N}$ given by $R_N\phi\coloneqq \left(\phi(t_1),\dots,\phi(t_N)\right)^{T}$ and the prolongation operator $P_N\colon Y_{N}\to \mathbb{F}_N\subset Y$ given by $P_N\Phi\coloneqq \sum_{i=1}^N\ell_i\Phi_i$, where $\Phi\coloneqq (\Phi_1,\dots,\Phi_N)^{T}$ with $\Phi_i\in \C^d$.
Note that $R_N P_N= I_{Y_{N}}$ and $P_N R_N =  L_N$, where $L_{N}\colon Y\to \mathbb{F}_N$ is the Lagrange interpolation operator relevant to $\Theta_{N}$.
Observe also that $\mathcal M_N$ is invertible iff for every $\Psi \in Y_{N}$ there exists a unique $\Phi\in Y_{N}$ s.t. $
\mathcal M_N\Phi=\Psi$.
Since $\mathcal M P_N\Phi \in C_\tau(\R, \C^d)$ (see \Cref{Assumption CBM}), by letting $\phi_N\coloneqq P_N\Phi$ it is not difficult to see that $R_N \mathcal M P_N\Phi$ is well-defined and $\Psi\coloneqq \mathcal M_N\Phi=R_N \mathcal M P_N\Phi=R_N \mathcal M \phi_N$. 
Applying $P_N$ to both sides gives $P_N\Psi=L_N\mathcal M\phi_N$.
Since $L_N \phi_N'=\phi_N'$, we get $L_N\mathcal M\phi_N=\phi_N'+L_{N}M\phi_N$ and therefore $\phi_N'=P_N\Psi- L_NM\phi_N$.
Adding $M\phi_N$ to both sides finally leads to\footnote{Note that here we see $L_N$ as an operator from $Y$ to $X$.}
\begin{equation}\label{eqpol}
\mathcal M\phi_N=P_N\Psi+(I_X-L_N)M\phi_N.
\end{equation}
Now, since $\phi_N\in Y$ and $\mathcal M$ is invertible,\footnote{Note that this is true only for the case of a trigonometric polynomial. In fact, in general, an algebraic polynomial is not a function of $X$.} there exists a unique $\psi\in X$ s.t. $\phi_N =\mathcal M^{-1}\psi$.\footnote{Note that $\psi$ may not be a trigonometric polynomial in general.}
Then it is not difficult to see that solving \eqref{eqpol} for $\phi_{N}$ is equivalent to solve $\psi=P_N \Psi+(I_X-L_N)M\mathcal M^{-1}\psi$ for $\psi$, which can be equivalently rewritten as $\left(I_X+\left(L_N-I_X\right)M\mathcal M^{-1}\right) \psi=P_N\Psi$.
We now show that $\mathcal M_N$ is invertible by proving that the latter admits a unique solution.
\begin{lemma}\label{lemmainv}
Under Assumptions \ref{AssFourier} and \ref{ass3} there exists a positive integer $\bar N$ s.t. for every $N\ge \bar N$, $\left(I_X+(L_N-I_X)M\mathcal M^{-1}\right)^{-1}$ exists bounded and
\begin{equation}\label{boundinvM}
\left\|\left(I_X+(L_N-I_X)M\mathcal M^{-1}\right)^{-1}\right\|_{X\leftarrow X}\le 2.
\end{equation}
Moreover, $\mathcal M_N$ is invertible and
$\mathcal M_N^{-1}=R_N\mathcal M^{-1}\left(I_X+(L_N-I_X)M\mathcal M^{-1}\right)^{-1}P_N.$
\end{lemma}
\begin{proof}
We use a proof technique similar to \cite[Proposition 3]{BredaKuniyaRipollVermiglio2020}. Observe that $M\mathcal M^{-1}\colon X\to Y$ is bounded (see the proof of \Cref{HZcompact}). Moreover, \cite[p.~221, Theorem 3.2.9]{Mastroianni} ensures that for $\phi\in Y$ and every $N\ge 3$ there exists a constant $K>0$  independent of $N$ s.t.
\begin{equation}\label{mastroiannibound}
\|(L_N-I_X)\phi\|_X \le K\frac{\log \lfloor{N/2}\rfloor}{\lfloor{N/2}\rfloor}\|\phi'\|_X\le K\frac{\log \lfloor{N/2}\rfloor}{\lfloor{N/2}\rfloor}\|\phi\|_{Y}.
\end{equation}
Thus
\begin{equation}\label{boundMM}
\left\| (L_N-I_X) M\mathcal M^{-1} \right\|_{X\leftarrow X}\to 0,\qquad N\to \infty,
\end{equation}
and the claim follows by the Banach perturbation Lemma \cite[p.~142, Theorem 10.1]{Kress1989}. 
\end{proof}
Let us observe that if $\Psi\coloneqq R_N \psi$ for $\psi \in C_\tau(\R, \C^d)$ then the unique solution $\phi_N\in \mathbb F_N$ of \eqref{eqpol} 
is given by $\phi_N=\mathcal M^{-1}\left(I_X+(L_N-I_X)M\mathcal M^{-1}\right)^{-1}L_N \Psi$.
Hence, by introducing the operator
\begin{equation}\label{invcontM}
\widehat{\mathcal M}_N^{-1}\colon C_\tau(\R, \C^d)\to Y,\qquad \widehat{\mathcal M}^{-1}_N\coloneqq \mathcal M^{-1}\left(I_X+(L_N-I_X)M\mathcal M^{-1}\right)^{-1}L_N,
\end{equation}
we have $\phi_N\coloneqq \widehat{\mathcal M}^{-1}_N\Psi$.
\begin{remark}
In the proof of \Cref{lemmainv}, the regularizing effect of the operator $\mathcal M^{-1}$ allows us to evaluate the interpolation error of a function of $Y$ with the $X$-norm, and then map the problem back to $Y$ via $\mathcal M^{-1}$. 
\end{remark}

Let us now observe that $\mathcal B$ is a multiplicative operator, thus with no regularizing effect on $X$ or $Y$. This brings in some difficulties in directly comparing $\mathcal K$ and $\mathcal K_N\coloneqq\mathcal{B}_{N}\mathcal{M}_{N}^{-1}$. Nevertheless, the relations $\sigma(\mathcal H)\setminus\{0\}=\sigma(\mathcal K)\setminus\{0\}$ and $\sigma(\mathcal H_N)=\sigma(\mathcal K_N)$ hold for the compact operator $\mathcal H\coloneqq \mathcal M^{-1}\mathcal B$ and the matrix $\mathcal H_N\coloneqq \mathcal M_N^{-1}\mathcal B_N$.\footnote{We remark that the operator $\mathcal H$ is actually the ``Next-Infection Operator'' considered in \cite{Wang2008} for the computation of $R_0$. See also \cite{Thieme2009}.}
Hence, to investigate the convergence of the eigenvalues of $\mathcal K_N$ to those of $\mathcal K$ we can investigate the convergence of the eigenvalues of $\mathcal H_N$ to those of $\mathcal H$.
Furthermore, as $\Range(\mathcal H)=Y$, we can restrict to $Y$. 

The operator ${P_N}\mathcal H_NR_N=\widehat {\mathcal M}_N^{-1}\mathcal BL_N\colon Y\to Y$ has the same nonzero eigenvalues with the same geometric and partial multiplicities of $\mathcal H_N$ \cite[Proposition 4.1]{BredaMasetVermiglio2012}. Since  $\Range(P_N)\subset \mathbb F_N$, ${P_N}\mathcal H_NR_N$ has the same nonzero eigenvalues with the same geometric and partial multiplicities and the same eigenfunctions of the operator $\widehat{\mathcal H}_N\coloneqq \widehat {\mathcal M}_N^{-1}\mathcal B\colon Y\to \mathbb F_N$ \cite[Proposition 4.3]{BredaMasetVermiglio2012}.
Since $\widehat{\mathcal H}_N-\mathcal H\colon Y\to Y$ is compact\footnote{The compactness of the natural immersion $Y\hookrightarrow  X$ ensures that the restriction to $Y$ of a bounded operator from $X$ to $Y$ is compact.} for every $N\ge0$, to apply the spectral approximation theory of \cite{Chatelin1981} it is enough to show that $\|\widehat{\mathcal H}_N-\mathcal H\|_{Y\leftarrow Y}\to 0$ as $N\to \infty$.
\begin{theorem}
Under Assumptions \ref{AssFourier} and \ref{ass3}, 
$\|\widehat{\mathcal H}_N-\mathcal H\|_{Y\leftarrow Y}\to 0$ as $N\to \infty$.
\end{theorem}
\begin{proof}
As $\widehat{\mathcal H}_N-\mathcal H=(\widehat{\mathcal M}^{-1}_N - \mathcal M^{-1})\mathcal B$, we investigate the behavior of $\widehat{\mathcal M}^{-1}_N-\mathcal M^{-1}$.
From \eqref{invcontM}, we have
\begin{align}\label{term1}
\widehat {\mathcal M}_N^{-1}-\mathcal M^{-1}= &\  
\mathcal M^{-1}\left(I_X+(L_N-I_X)M\mathcal M^{-1}\right)^{-1}(L_N-I_X)\\ \label{term2}
&+\mathcal M^{-1}\left((I_X+(L_N-I_X)M\mathcal M^{-1})^{-1}- I_X\right).
\end{align}
Direct computations show that
\begin{gather*}
(I_X+(L_N-I_X)M\mathcal M^{-1})^{-1}-I_X=
(I_X+(L_N-I_X)M\mathcal M^{-1})^{-1}(I_X-L_N)M\mathcal M^{-1}.
\end{gather*}
Hence,  \eqref{boundinvM} together with \eqref{term1}-\eqref{term2} gives
\begin{align}\label{boundnormaMper}
\|\widehat{\mathcal M}_N^{-1}-\mathcal M^{-1}\|_{Y\leftarrow Y}
\le K\bigg(\left\|
(L_N-I_{X})\right\|_{X\leftarrow Y}+\|(L_N- I_X)\mathcal M\mathcal M^{-1}\|_{X\leftarrow Y}\bigg)
\end{align}
for $N$ sufficiently large, where $K\coloneqq 2\|\mathcal M^{-1}\|_{Y\leftarrow X}$.
Then, the first term on the right-hand side of \eqref{boundnormaMper} vanishes thanks to \eqref{mastroiannibound}, while the second term vanishes thanks to \eqref{boundMM}.
Finally, under \Cref{ass3}, $\mathcal B$ is bounded and the claim follows by observing that $\|\widehat{\mathcal H}_N-\mathcal H\|_{Y\leftarrow Y}\le \left\|\widehat {\mathcal M}_N^{-1}-\mathcal M^{-1}\right\|_{Y\leftarrow Y}\mathcal \|\mathcal B\|_{Y\leftarrow Y}$.
\end{proof}

We conclude with the main convergence result.
\begin{theorem}\label{teospecper}
Under Assumptions \ref{AssFourier} and \ref{ass3}, let $\lambda\in\C$ be an isolated nonzero eigenvalue of $\mathcal H$ with finite algebraic multiplicity $m$ and ascent $l$, and let $\Delta$ be a neighborhood of $\lambda$ s.t. $\lambda$ is the sole eigenvalue of $\mathcal H$ in $\Delta$. Then there exists a positive integer $\bar N$ s.t. for every  $N\ge \bar N$ $\widehat{\mathcal H}_N$ has in $\Delta$ exactly $m$ eigenvalues $\lambda_{N, i}$, $i=1,\dots, m$ (counting multiplicities), and 
\begin{align*}
\max_{i=1, \dots, m}|\lambda_{N, i}-\lambda|=O\big(\varepsilon_N^{1/l}\big),
\end{align*}
where
\begin{equation}\label{epsilon}
\varepsilon_N\coloneqq \|\mathcal {\widehat H}_N-\mathcal H\|_{Y\leftarrow E_{\lambda}}
\end{equation}
and $E_{\lambda}$ is the generalized eigenspace of $\lambda$. 
Moreover, for any $i=1,\dots, m$ and for any eigenfunction $\psi_{N, i}$ of $\mathcal {\widehat H}_N$ relevant to $\lambda_{N,i}$ s.t. $\|\psi_{N,i}\|_Y=1$, we have
$\dist(\psi_{N,i},\  \ker(\lambda I_Y-\mathcal H))=O(\varepsilon_N^{1/l}),$
with $\dist$ the distance in $Y$ between an element and a subspace. Finally, $\varepsilon_N=O(\rho_{N})$ for
\begin{equation}\label{fbound}
\rho_N\coloneqq \begin{cases}
N^{-(s+1)}\log N \quad&\text{under \Cref{ass3}\ref{ass31}},\\
N^{-r}\quad&\text{for every integer $r\ge 1$ under \Cref{ass3}\ref{ass32}},\\
p^{-N}\quad&\text{for some constant $p>1$ under \Cref{ass3}\ref{ass33}}.
\end{cases}
\end{equation}
\end{theorem}
\begin{proof}
Most of the proof follows from \cite[Proposition 2.3 and Proposition 4.1]{Chatelin1981}. It is left to prove \eqref{fbound}, for which it is enough to show that 
\begin{equation}\label{maxbound}
\max_{i=1,\dots, m}\|(\widehat{\mathcal H}_N-\mathcal H)\psi_{\lambda,i}\|_{Y}=O(\rho_N)
\end{equation}
for $\{\psi_{\lambda,i}\}_{i=1}^m$ a basis of $ E_\lambda$. Then the thesis will follow by proceeding as in \cite[Proposition 4.9]{liessi2018}.
From \eqref{boundnormaMper} we get
\begin{equation*}
\|(\widehat{\mathcal H}_N-\mathcal H)\psi_{\lambda, i}\|_{Y}
\le  K\left( \|(L_N - I_X)B\psi_{\lambda, i}\|_{X}+\|(L_N-I_X)M\mathcal M^{-1}B\psi_{\lambda, i}\|_{X}\right),
\end{equation*} 
for $K\coloneqq 2\|\mathcal M^{-1}\|_{Y\leftarrow X}$ and
\begin{enumerate}
\item $B\psi_{\lambda, i},\ M\mathcal M^{-1}B\psi_{\lambda, i} \in \Lip^{s}_\tau(\R, \C^{d})$ under \Cref{ass3}\ref{ass31}, 
\item $B\psi_{\lambda, i},\ M\mathcal M^{-1}B\psi_{\lambda, i} \in C_\tau^{\infty}(\R, \C^{d})$ under \Cref{ass3}\ref{ass32},
\item $B\psi_{\lambda, i},\ M\mathcal M^{-1}B\psi_{\lambda, i}$ are analytic under \Cref{ass3}\ref{ass33}.
\end{enumerate}
Then $\|(L_N-I_X)B\psi_{\lambda, i}\|_{X}\le \tau \|(L_N-I_Y)B\psi_{\lambda, i}\|_{\infty}$
and $\|(L_N-I_X)M\mathcal M^{-1}B\psi_{\lambda, i}\|_{X} $ $\le \tau\|(L_N-I_Y)M\mathcal M^{-1}B\psi_{\lambda, i}\|_{\infty}$ hold and,
under \Cref{ass3}\ref{ass31}, standard results in approximation theory and Jackson's type theorems \cite[section 1.1.2]{Rivlin1981} ensure that $$\|(L_N-I_Y)B\psi_{\lambda, i}\|_{\infty}\le (1+\Lambda_N)  E_N(B\psi_{\lambda, i})=O(\rho_{N})$$
and 
$$\|(L_N-I_Y)M\mathcal M^{-1}B\psi_{\lambda, i}\|_{\infty}\le (1+\Lambda_N)  E_N(M\mathcal M^{-1}B\psi_{\lambda, i})=O(\rho_{N}),$$
where $\Lambda_N$ is the Lebesgue constant relevant to $\Theta_N$, which equals $O(\log N)$ under \Cref{AssFourier} \cite[p.~211, eq.~3.2.32]{Mastroianni}, and $E_N(f)$ is the best uniform approximation error  on $\mathbb F_N$ for a continuous periodic function $f$.
The bounds under \Cref{ass3}\ref{ass32} and \Cref{ass3}\ref{ass33} follow similarly from \cite[Theorem 4.2]{wright2015}.
Therefore, \eqref{maxbound} holds and the proof is complete. 
\end{proof}

\subsection{Convergence of Chebyshev collocation}
We first derive a characteristic equation for the exact generalized eigenvalue problem \eqref{gep}. Since we are interested in the (positive) spectral radius of $\mathcal K$, we assume $\lambda\ne 0$, so that \eqref{gep} can be rewritten as
\begin{equation}\label{gep1}
\frac{1}{\lambda} \mathcal B\phi =\mathcal M\phi.
\end{equation}
By setting $A^{(\lambda)}\coloneqq B/\lambda-M$ for brevity, we observe that $(\lambda, \phi)$ solves \eqref{gep1} iff $\phi$ solves the periodic boundary value problem
\begin{equation}\label{pbvp}
\left\{\setlength\arraycolsep{0.1em}\begin{array}{rlll} 
\phi'(t)&=&A^{(\lambda)}(t)\phi(t),&\qquad t\in [0, \tau], \\[0mm]
\phi(0)&=&\phi(\tau). &
\end{array} 
\right.
\end{equation}
Now, the family of IVPs
\begin{equation}\label{ivp}
\left\{\setlength\arraycolsep{0.1em}\begin{array}{rlll} 
\phi'(t)&=&A^{(\lambda)}(t)\phi(t),&\qquad t\in [0, \tau], \\[0mm]
\phi(0)&=&\alpha, &
\end{array} 
\right.
\end{equation}
for $\alpha\in \C^d$ admits a unique solution $\phi(t; \lambda, \alpha)$ for every $\alpha$, which can be expressed as $$\phi(t; \lambda, \alpha) = U^{(\lambda)}(t, 0)\alpha,$$ for $U^{(\lambda)}$ the PMS of \eqref{ivp}. With abuse of notation we write $U^{(\lambda)}(t)=U^{(\lambda)}(t, 0)$. Since a solution of \eqref{ivp} solves \eqref{pbvp} iff $\alpha=\phi(0; \lambda, \alpha) = U^{(\lambda)}(\tau)\alpha$,
we arrive at the characteristic equation
\begin{equation}\label{chareq}
c(\lambda) = 0,
\end{equation}
for the characteristic function
$c(\lambda) \coloneqq \det \left(I_{\C^d} - U^{(\lambda)}(\tau)\right)$.

Similarly, we derive the discrete counterpart of \eqref{chareq} by observing that $(\lambda, \Phi)\in \C\times Y_{N}$ satisfies  $\mathcal B_{N}\Phi/\lambda =\mathcal M_{N}\Phi$
iff the polynomial $\phi_N$ interpolating $\Phi$ at the Chebyshev nodes satisfies
\begin{equation}\label{pbvpC}
\left\{\setlength\arraycolsep{0.1em}\begin{array}{rlll} 
\phi_N'(t_i)&=&A^{(\lambda)}(t_i)\phi_N(t_i),&\qquad i=1,\dots, N, \\[0mm]
\phi_N(0)&=&\phi_N(\tau). &
\end{array} 
\right.
\end{equation}
Below, we prove that the auxiliary collocation problem
\begin{equation}\label{ivpcoll}
\left\{\setlength\arraycolsep{0.1em}\begin{array}{rlll} 
p_N'(t_i)&=&A^{(\lambda)}(t_i)p_N(t_i),&\qquad i=1,\dots, N, \\[0mm]
p_N(0)&=&\alpha. &
\end{array} 
\right.
\end{equation}
for \eqref{ivp} has a unique solution for every $\alpha\in \C^d$. Hence, we can define $U_N^{(\lambda)}\colon [0, \tau]\to \C^{d\times d}$ s.t. 
$$p_{N}(t; \lambda, \alpha)=U_N^{(\lambda)}(t)\alpha.$$
Since a solution of \eqref{ivpcoll} solves \eqref{pbvpC} iff $\alpha=p_{N}(0; \lambda, \alpha) = U_N^{(\lambda)}(\tau)\alpha$,
we arrive at the discrete characteristic equation
\begin{equation}\label{chareqdisc}
c_N(\lambda) = 0,
\end{equation}
for the characteristic function
$c_N(\lambda) \coloneqq \det \left(I_{\C^d} - U_N^{(\lambda)}(\tau)\right)$.

The convergence of the eigenvalues is then investigated by studying the error $|c_N(\lambda)-c(\lambda)|$
in a neighborhood of a spectral value of $\mathcal K$.
To this aim let $B(\lambda, z)$ denote the closed ball in $\C$ of center $\lambda$ and radius $z>0$.
Then, for every $\lambda^*\in \C$ and $z>0$ there exists $C_1=C_1(\lambda^*, z) >0$ s.t.  
\begin{equation}\label{boundeqcarper}
|c_N(\lambda)-c(\lambda)|\le C_1\left\|U_N^{(\lambda)}(\tau) - U^{(\lambda)}(\tau)\right\|_{\C^d\leftarrow \C^d}=C_1\sup_{
|\alpha|_{\C^d}=1}
|e_N(\tau; \lambda, \alpha)|_{\C^d}
\end{equation}
since $\det(A)$ is a locally Lipschitz function w.r.t. $A\in \C^{d\times d}$ and where $$e_{N}(\cdot; \lambda, \alpha)\coloneqq p_N(\cdot; \lambda, \alpha)-\phi(\cdot; \lambda, \alpha)$$ is the collocation error of \eqref{ivp} through \eqref{ivpcoll}. Thus we need a bound for $| e_N(\tau; \lambda, \alpha)|_{\C^d}$. First we prove that \eqref{ivpcoll} is well posed.
Under \Cref{ass3}, the solutions of \eqref{ivp}, and consequently $e_N$, are  continuously differentiable.
Hence we can restrict to $C([0, \tau], \C^d)$.
For the sake of readability, hereafter we avoid explicitly writing the dependence of $\phi, p_N$ and $e_N$ on $\lambda, \alpha$ when unnecessary.
Let $L_{N-1}\colon C([0, \tau], \C^d)\to C([0, \tau], \C^d)$ denote the Lagrange interpolation operator relevant to $\Theta_N$. 
Since $L_{N-1}p_N'=p_N'$, it follows from \eqref{ivp} and \eqref{ivpcoll} that $e_N'$ satisfies the functional equation
\begin{equation}\label{colleqextrema}
e_{N}'\coloneqq L_{N-1}A^{(\lambda)}Ve'_{N}+r_N,
\end{equation}
where $V\colon C([0, \tau], \C^d)\to C([0, \tau], \C^d)$ is the Volterra operator $(V\eta)(t)=\int_0^t \eta(\xi)\dd\xi$
and
$r_N\coloneqq (L_{N-1}-I_{C([0, \tau], \C^d)})A^{(\lambda)}\phi$.
Now we show that \eqref{colleqextrema} has a unique solution, and we derive a bound to $|e_N(\tau)|\coloneqq |(Ve_N')(\tau)|$ in terms of $r_N$. To this aim, we need $\tilde\Lambda_{N}/N$ to vanish as $N\to\infty$ for the Lebesgue constant $\tilde \Lambda_N\coloneqq \|L_{N-1}\|_{C([0, \tau], \C^d)\leftarrow C([0, \tau], \C^d)}$ relevant to $\Theta_{N}$.
Unfortunately, $\tilde\Lambda_{N}/N$ is only known to be bounded under \Cref{AssCheb} when $C([0, \tau], \C^d)$ is endowed with the uniform norm \cite[section 4.2]{Mastroianni}.
To overcome the issue we resort to the (weaker) weighted uniform norm $\|\eta\|_{w, \infty}\coloneqq \|\eta\cdot w\|_\infty$, where $w(t)\coloneqq (2/\tau)^2\sqrt{(\tau-t)t}$. In fact, by denoting with $C_w$ the space $C([0, \tau], \C^d)$ endowed with $\|\cdot \|_{w, \infty}$, one has $\Lambda_N\coloneqq \|L_{N-1}\|_{C_w\leftarrow C_w}=O(\log N)$ \cite[p.~274, Theorem 4.3.2]{Mastroianni}.
Additionally, we consider the space $\Lip_w([0, \tau], \C^d)$, which is the space of $\C^d$-valued Lipschitz continuous functions on $[0, \tau]$ 
endowed with the $w$-weighted norm $\|\eta\|_{\Lip_w^{p}}\coloneqq \| \eta w\|_{\infty}+L_w(\eta)$, where $L_w(\eta)$ denotes the Lipschitz constant of $\eta$ w.r.t. the norm $\|\cdot \|_w$. Then we have the following result.
\begin{proposition}\label{collerrper}
Under Assumptions \ref{AssCheb} and \ref{ass3} there exists a positive integer $\bar N=\bar N(\lambda)$ s.t., for every integer $N\ge \bar N$,  \eqref{colleqextrema} has a unique solution $e_N'$ in $C_w$ and
\begin{equation}\label{boundonen}
|e_N(\tau)| \le 2 \pi|\alpha| \varepsilon_N\left\| \left(I_{C_w}- A^{(\lambda)}V\right)^{-1}\right\|_{C_w\leftarrow C_w}, 
\end{equation}
where
\begin{equation}\label{fboundCheb}
\varepsilon_N\coloneqq \begin{cases}
O\left(N^{-(s+1)}\log N\right) \quad&\text{under \Cref{ass3}\ref{ass31}},\\
O\left(N^{-r}\right)\quad&\text{for every integer $r\ge 1$ under \Cref{ass3}\ref{ass32},}\\
O\left(p^{-N}\right)\quad&\text{for some constant $p>1$ under \Cref{ass3}\ref{ass33}}.
\end{cases}
\end{equation}
\end{proposition}
\begin{proof}
We adapt the arguments of \cite[Proposition 3]{BredaKuniyaRipollVermiglio2020} by proving that there exists a positive integer $\bar N:=\bar N(\lambda)$ s.t. \eqref{colleqextrema}
admits a unique solution for every $N\ge\bar N$.
Let us rewrite \eqref{colleqextrema} 
as $(I_{C_w}-L_{N-1}A^{(\lambda)}V) e_{N}' =r_N$
and observe that
$(I_{C_w}-L_{N-1}A^{(\lambda)}V) = (I_{C_w}-A^{(\lambda)}V) + (I_{C_w}-L_{N-1})A^{(\lambda)}V$.
It is not difficult to see that $I_{C_w}-A^{(\lambda)}V$ is invertible with bounded inverse in $C_w$.
Furthermore, since $\Range(A^{(\lambda)}V)\subset \Lip_w$ and $A^{(\lambda)}V\colon C_w \to \Lip_w$ is bounded, \Cref{AssCheb}, $\Lambda_N=O(\log N)$ and \cite[p.~271, eq.~4.33]{Mastroianni}
guarantee that there exists a positive constant $K$ independent of $N$ s.t. \begin{align}\label{Mastroianniext}
\left\|(I_{C_w}-L_{N-1})A^{(\lambda)}V\right\|_{C_w\leftarrow C_w}
\le K \|A^{(\lambda)}V\|_{\Lip_w\leftarrow C_w} \frac{\log N}{N}\to 0,\quad N\to \infty.
\end{align}
Hence, the Banach perturbation Lemma \cite[p.~142, Theorem 10.1]{Kress1989} ensures that there exists a positive integer $\bar N=\bar N(\lambda)$ s.t. $I_{C_w}-{L}_{N-1}A^{(\lambda)}V$ is invertible in $C_w$ and $\|e_N'\|_{w, \infty}\ \le 2\| \left(I_{C_w}- A^{(\lambda)}V\right)^{-1}\|_{C_w\leftarrow C_w}\|r_N\|_{w, \infty}$
holds for every $N\ge \bar N$. From \cite[p.~271, eq. 4.33]{Mastroianni} and standard results in approximation theory, we get $\|r_N\|_{w, \infty}\le
\left(1+\Lambda_N\right) E_{N-1}(A^{(\lambda)}\phi)$,
where $E_{N-1}(f)$ is the best approximation error  on $\mathbb P_{N-1}$ for a continuous function $f$. Then
$\|r_N\|_{w,\infty}=|\alpha|\varepsilon_N$ follows from \Cref{ass3} and $\Lambda_N=O(\log N)$ since $r_N$ depends linearly on $\alpha$, and estimates of $E_{N-1}(A^{(\lambda)}\phi)$ can be obtained by applying Jackson's type theorems \cite[section 1.1.2]{Rivlin1981}, \cite[chapter 7 and 8]{Trefethen2013}. Finally, \eqref{boundonen} follows from $|e_N(\tau)|_{\C^d} \le \|e_N'\|_{L^1([0, \tau], \C^d)}\le \|w^{-1}\|_{L^1([0, \tau], \C)}\| e_N' \|_{w, \infty}=\pi\| e_N' \|_{w, \infty}$.
\end{proof}

We now state the main convergence result, whose proof follows from \eqref{boundeqcarper} and \Cref{collerrper} by eventually applying Rouche's theorem as in \cite[section 4]{BredaKuniyaRipollVermiglio2020}.
\begin{theorem}\label{boundpercheb}
Let $\lambda^*$ be a zero of $c$ in \eqref{chareq} with algebraic multiplicity $m$ and $z>0$ be s.t. $\lambda^*$ is the only zero of $c$ in $\bar B(\lambda^*, z)$. Then, there exists a positive integer $\bar N$ s.t., for every integer $N\ge \bar N$, $c_N$ in \eqref{chareqdisc} has in $\bar B(\lambda^*, z)$ exactly $m$ zeros $\lambda_{N,i}$, $i=1,\dots, m$, counting their multiplicities, and $\max_{i=1,\dots,m} |\lambda^*-\lambda_{N,i}|=O(\varepsilon_N^{1/m})$,
where $\varepsilon_N$ is given in \eqref{fboundCheb}.
\end{theorem}

\section{Numerical results}\label{results}
In this section, we experimentally validate the theoretical results of \cref{convanal} and discuss applications to epidemiological models. We propose two test examples. The first one is a model obtained from the linearization of a $d$-dimensional Susceptible--Infectious--Removed (SIR) model with seasonality in contact rates around the Disease-Free Equilibrium (DFE). Under suitable assumptions on the model coefficients, we derive an explicit expression
for $R_0$, and use it to numerically illustrate the convergence order of the numerical approximation 
of $R_0$ for both the Fourier and Chebyshev methods (for the latter, we use both piecewise and nonpiecewise versions), which in turn depends on the regularity of the model parameters, see \Cref{teospecper},  \Cref{collerrper}, and \Cref{boundpercheb}. Moreover, we investigate computational times for large $d$ (i.e., $d=10, 50,100$), comparing the proposed scheme to the Monodromy Operator Method (MOM), which is based on the numerical solution of \eqref{nonlineareq} \cite{Bacaer2007,Wang2008}.\footnote{We do not make comparisons with methods based on the discretizion of \eqref{ngoeig}, as it is clear that these are computationally more challenging than the method here proposed.} 
The MOM is applied by choosing as initial guess the value of $R_0$ for the averaged problem. At each iteration, the $\lambda$-dependent ODE system is solved over one period using the MATLAB solvers \textsf{ode45} or \textsf{ode113}. Finally, the nonlinear equation is solved by using the MATLAB bult-in function \textsf{fzero}.

The second example has $d=5$, and is obtained by linearizing a vector-host model around the Disease-Free Periodic Solution (DFPE). We use it to illustrate how the proposed method can be easily adapted to compute TRNs. In this case, analytic expressions for the reproduction numbers are not available. Hence, the errors are computed with respect to reference values obtained with the Fourier method using $N=151$.\footnote{We always use the Fourier method with odd $N$. Numerical experiments (not included here) show the possible presence of a spurious dominant eigenvalue for $N$ even.}

In the following, for computational efficiency, we always take $\tau=1$, as it is easy to see that, for $\tau \in(0,+\infty)$, \eqref{gep} can always be scaled as $\hat B( t)\phi(t)=\lambda[\phi'(t)+\hat M(t)\phi(t)]$, for $(\lambda, \phi)\in \R\times AC_1(\R, \R^d)$, where $\hat B(t):=\tau B(\tau t)$, $\hat M(t):=\tau M(\tau t)$, for $t\in [0,1]$. 

The tests are performed with MATLAB R2020b on a
machine with Intel Core i3-1005G1 CPU and 8 Gb of RAM, running Windows 10.

\subsection{A multi-group SIR model with seasonality in contact rates}\label{SectionSIR}
Inspired by \cite[section 8.8 and 8.9]{IannelliPugliese}, we consider an 
SIR model for the spread of an infectious disease in a population which is split into $d\geq 1$ groups with seasonal contact rates. In the model, we include demography and neglect waning immunity and disease-induced deaths. 
Denoting by $S_j(t),  I_j(t)$ and $R_j(t)$ the number of susceptible, infectious and recovered individuals, respectively, at time $t$, and by $P_j(t)=S_j(t)+I_j(t)+R_j(t)$ the total size of group $j=1 \dots d$, the model reads
\begin{equation}\label{linVB}
\left\{\setlength\arraycolsep{0.1em}
\begin{array}{rl} 
S_j'(t)&= \mu_j(P_j(t)-S_j(t)) -S_j(t) \displaystyle\sum_{k=1}^d \beta_{j k}(t)\cfrac{I_k(t)}{P_k(t)},\\
I_j'(t)&= S_j(t) \displaystyle\sum_{k=1}^d \beta_{j k}(t)\cfrac{I_k(t)}{P_k(t)} - (\gamma+\mu_j) I_j(t),\\[6mm]
R_j'(t)&= \gamma I_j(t)-\mu_j R_j(t),
\end{array} 
\right.
\end{equation}
where $1/\mu_j$ is the average life span of group $j$, $1/\gamma$ is the average infectious period, and
$\beta_{jk}(t)= qc_j(t)c_k(t)P_k/(\sum_{h=1}^d c_h(t) P_h)$, $1 \leq j$, $k \leq d$, is the   transmission rate from an infectious individual of group $k$ to a susceptible individual of group $j$, with $q$ the probability of infection per contact, and $c_j(t)> 0$ the $\tau$-periodic contact rate of  individuals of group $j$. 
 Note that $P_j'= 0$ for every $j=1,\dots, d$, so $P_j$ is constant over time. 
By defining $s_j:=S_j/P_j$, $i_j:=I_j/P_j$ and $r_j=R_j/P_j$, $j=1,\dots, d$, the resulting model admits the DFE $s_j^*=1,\ i_j^*= r_j^*=0$, $j=1,\dots, d$. The linearized equations for the infectious individuals around the DFE 
can be recast as in \eqref{ODEBM} by taking
$\left[B(t)\right]_{j,k}=\beta_{jk}(t)$, $j,k=1,\dots, d$, and $M(t)\equiv\diag\left(\mu_1+\gamma,\dots, \mu_d+\gamma \right)$\footnote{Being $M$ time independent, the associated evolutionary system $\{V(t, s)\ :\ t, s\in \R,\ t\ge s\}$ is explicitly given by $V(t,s)= e^{-M(t-s)}$ and the NGO $\mathcal{K}$ can be explicitly derived. 
}.

In the following, we assume 
$\mu_j\equiv \mu>0$ and $c_j(t)= P_j f(t)$ for some nonnegative $f\in L^\infty_\tau(\R, \R_{>0})$,  $j=1,\dots, d$.
Under these assumptions, we have
\begin{equation}\label{R0SIR}
R_0=\cfrac{1}{\gamma+\mu}\sum_{k=1}^d
\cfrac{1}{\tau}\int_0^\tau [B(\theta)]_{k,k}\dd \theta,
\end{equation}
with generalized eigenfunction
\begin{equation*}
[\phi(t)]_j=\exp\left(\cfrac{1}{R_0}\sum_{k=1}^d\cfrac{1}{\tau}\int_0^t [B(\theta)]_{k,k}\dd \theta - (\gamma+\mu)t\right)P_j,\qquad j=1,\dots, d,\end{equation*}
and eigenfunction $\psi=B\phi/R_{0}$.
Notice that for $d=1$, one recovers the expression derived in \cite[section 5]{Bacaer2006}. We use \eqref{R0SIR} to investigate the convergence of the proposed method. We assume parameters inspired by measles and adapted from \cite[section 8.8 and 8.9]{IannelliPugliese}: $\tau =1$ (years), $1/\gamma= 7$ (days), $q= 60\%$,  $1/\mu=82$ (years), and $P_j\equiv P=1000$,\  $j=1,\dots, d$. For the readers' convenience, we summarize the parameters in \Cref{TableSIR}.
\begin{table}
\footnotesize\caption{Parameters for the numerical experiments of \cref{SectionSIR}, adapted from \cite[chapter 8]{IannelliPugliese}.}
\label{TableSIR}
\begin{center}
\begin{tabular}{c  c c }
\rowcolor{gray!20}
Parameter & interpretation & value \\[0.5mm]
\hline
$\tau$ & time period & 1 (years)\\[0.5mm]
$q$ & probability of infection per contact & 60\% \\[0.5mm]
$1/\gamma$ & average infectious period & 7 (days) \\[0.5mm]
$1/\mu$ & average life span (constant among groups) & 82 (years)\\[0.5mm]
$P$ & population size (constant among groups) & $1000$ \\[0.5mm]
\hline
\end{tabular}
\end{center}
\end{table}

\noindent Finally, inspired by \cite{buonomo2018seasonality}, we consider the following choices for $f$:
\begin{enumerate}[(F1)]
\item \label{F1} $f(t)= 1+\delta \sin(2\pi t/\tau-\eta),\quad \delta,\eta\in [0, 1)$,\smallskip
\item \label{F2} $f(t)= \kappa + \kappa\frac{\left(0.68 +\sin(2\pi t/\tau)\right)}{\left(1+2/3 \sin(2\pi t/\tau)\right)},\quad \kappa:=\left(5/2 - (123\sqrt{5})/250\right)^{-1}$,\smallskip
\item \label{F3} $f(t) =0.2+ \mathbbm 1_{[0, 0.4]\cup[0.6, 1]}(t/\tau)$.
\end{enumerate}
For these choices, $R_0\approx 11.504158733340011$.
Note that, under \ref{F3} $f$ has jump discontinuities at $\tau=0.4$ and $\tau=0.6$. In this case, we resort to the piecewise version of the Chebyshev method.
\begin{figure}[h]
\begin{center}
\includegraphics[width=1.\textwidth]{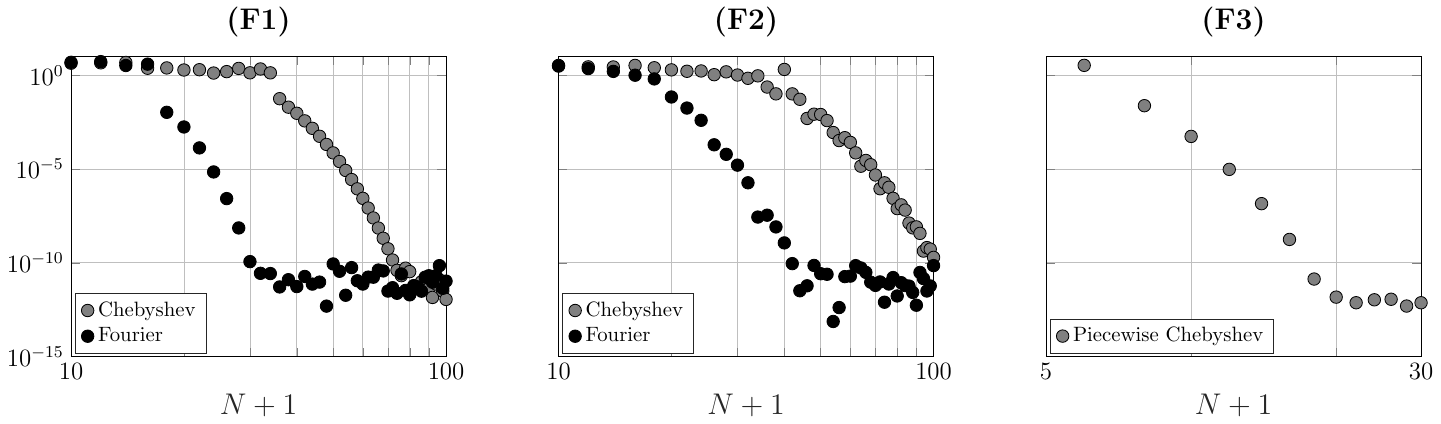}
\caption{\Cref{SectionSIR}: log-log plot of the error on $R_0$ with the Fourier method (black dots) and the Chebyshev method (gray dots) for increasing (odd) $N$ under \textsc{(F1)}  (left), \textsc{(F2)} (center) and \textsc{(F3)} (right), with $d=1$, $\delta=0.6$ and $\eta=0.2$. In the right panel we use the piecewise version of the Chebyshev method.}\label{SIRconv}
\end{center}
\end{figure}
\begin{figure}[h]
\begin{center}
\includegraphics[width=.75\textwidth]{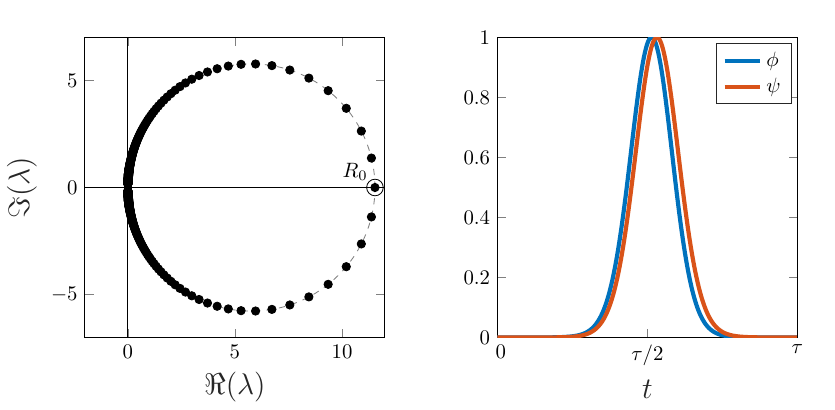}
\caption{\Cref{SectionSIR}: computed spectum (left) and eigenfunctions relevant to $R_0$, $\phi$ and $\psi$, (right) for the Fourier method with $N=501$ under \textsc{(F1)}, for $d=1$, $\delta=0.6$ and $\eta=0.2$.}\label{SIRspec&eig}
\end{center}
\end{figure}
In the left and center panels of \cref{SIRconv}, infinite order of convergence is observed for both the Fourier and the Chebyshev method under \ref{F1} and \ref{F2}, which is in accordance with the results of \cref{convanal}, as the model coefficients are real analytic. In particular, we observe that while a small $N+1$ (say 30 or 40) is enough for the Fourier method to obtain errors of the order of $10^{-10}$, for the Chebyshev method larger choices of $N$ are required to recover similar results. The right panel shows infinite order of convergence for the piecewise version of the Chebyshev method under \ref{F3}. 
\Cref{SIRspec&eig} shows the computed spectrum of the NGO and the eigenfunctions relevant to $R_0$ (normalized as $\|\phi\|_{\infty}=1$) computed with Fourier using $N=501$ under \ref{F1}, for $d=1$, $\delta = 0.6$ and $\eta =0.2$. Note that the results in the left panel are in agreement with those of \cite[chapter 7]{Bacaer2021b}, i.e., for $d=1$, all the eigenvalues of the NGO in the complex plane lie on the circle of diameter $R_0$ and centered at $(R_0/2, 0)$.
\begin{table}
\footnotesize\caption{\Cref{SectionSIR}: comparison between the Fourier method and the MOM under \textsc{(F1)} for increasing values of $d$, by varying $N$ for the Fourier method, the ODEs solver for the MOM,  with $\delta=0.6$, $\eta=0.2$. For the Fourier method, we compute the dominant eigenvalue only with the MATLAB built-in function \textsf{eigs$(\mathcal B_N, \mathcal M_N)$}. 
}
\label{SIRexpComp}
\begin{center}
\begin{tabular}{c c c r l l}
\rowcolor{gray!20}
method & $N$  & solver & $d$ & absolute error & CPU time (s) \\
[0.5mm]
\hline
Fourier  & 25 & - & 10 & $2.67\times 10^{-7}$ & $1.42\times 10^{-2}$ \\  
&  & - & 50 & $2.67\times 10^{-7}$ & $1.23\times 10^{-1}$ \\
   &  & - & 100 & $2.67\times 10^{-7}$ & $4.57\times 10^{-1}$ \\
  & 35 & - & 10 & $1.24\times 10^{-11}$
 & $2.48\times 10^{-2}$ \\
 &  & - & 50 &  $9.47\times 10^{-12}$ & $2.29\times 10^{-1}$\\
   &  & - & 100 & $9.57\times 10^{-12}$ & $1.30\times 10^{0}$ \\
\hline
MOM  & - & ode45 & 10 & $4.09\times 10^{-4}$ & $3.34\times 10^{-1}$  
\\ 
& - &  & 50 & $5.81\times 10^{-4}$ & $1.90\times 10^{0}$ \\
   & - &  & 100 & $6.17\times 10^{-4}$ & $1.23\times 10^1$ \\
  & - & ode113 & 10 & $6.73\times 10^{-4}$ & $2.45\times 10^{-1}$ \\  
  & - &  & 50 & $9.41\times 10^{-4}$ & $1.65\times 10^{0}$ \\
   & - &  & 100 & $5.91\times 10^{-4}$ & $9.05\times 10^{0}$ \\   
\hline
\end{tabular}
\end{center}
\end{table}

Finally, in \Cref{SIRexpComp}, we compare the Fourier method with the MOM, both in terms of accuracy and CPU time, for increasing values of $d$, under \ref{F1}. In particular, for the Fourier method we compute the real dominant eigenvalue only through the MATLAB function \textsf{eigs$(\mathcal B_N, \mathcal M_N)$}. Note that the Fourier method generally provides smaller absolute errors for $R_0$ requiring lower CPU times (for the MOM, the absolute errors are always of the order of $10^{-4}$, which is due to the accuracy of the MATLAB solver \textsf{fzero}).
Thus, the method here proposed not only has the advantage of being simple to implement and,  differently from the MOM, to allow to compute the whole spectrum of the NGO and the eigenfunctions of both $\mathcal K$ and $\mathcal H$ (where the latter is defined as in \Cref{subconvFourier}), but it also generally allows one to obtain more accurate results with lower computational times.

\subsection{The TRN for a vector-host model}\label{SectionType}
In epidemiology, the TRN $T$ for a specific host type represents the expected number of secondary cases in individuals of a certain type produced by one typical infected individual of the same type during its entire infectious period, either directly or through chains of infection passing through any sequence of the other types \cite{HEESTERBEEK20073, Heesterbeek2003}. See \cite{bacaer2012model, Inaba2012} for the definition of the TRN  in periodic environments.
In this section, we show how to adapt the method presented in \cref{method} to compute TRNs in, e.g., vector-host models.  Following \cite{Bacaer2007, heesterbeek1995threshold}, we consider a model for the spread of a mosquito-transmitted disease with seasonal fluctuations in the vector population. 
We ignore both human demography and waning immunity, and we neglect disease-induced death in both humans and vectors.
Inspired by \cite{feng2019modelling, poletti2011}, we assume that, upon infection, both humans and mosquitoes undergo a latent period (exposed phase).  Then, exposed humans become infectious, and either develop symptoms with probability $\omega \in [0,1]$, or become asymptomatic (we assume that there is no transition between the asymptomatic and symptomatic phase). Let us denote with $S_H(t),\ E_H(t),\ A_H(t),\ I_H(t)$ and $R_H(t)$ the number of susceptible, exposed, asymptomatic, symptomatic and removed humans, respectively, at time $t\in \R$,  by $S_V(t),\ E_V(t)$ and $I_V(t)$ the number of susceptible, exposed and infectious mosquitoes, respectively, at time $t\in \R$, and by $P_H=S_H+E_H+A_H+I_H+R_H$ and $P_V=S_V+E_V+I_V$ the total population size of humans and mosquitoes, respectively. The model couples the following system of ODEs for the humans\begin{equation}\label{VBHsystemH}
\left\{\setlength\arraycolsep{0.1em}\begin{array}{rl} 
S_H'(t)&= - b q_{H\leftarrow V} I_V(t)\cfrac{S_H(t)}{P_H(t)},\\[1.5mm]
E'_H(t)&= b q_{H\leftarrow V}I_V(t)\cfrac{S_H(t)}{P_H(t)} - \nu_H E_H,\\[1.5mm]
A_H'(t)&=  (1-\omega)\nu_H E(t) -\gamma A_H(t),\\[1.5mm]
I_H'(t)&=  \omega\nu_H E(t) -\gamma I_H(t),\\[1.5mm]
R_H'(t)&= \gamma\left(A_H(t)+I_H(t)\right),
\end{array} 
\right.
\end{equation}
with the following system for the mosquitoes
\begin{equation}\label{VBVsystemH}
\left\{\setlength\arraycolsep{0.1em}\begin{array}{rl} 
S_V'(t)&= \Lambda(t) -  b q_{V\leftarrow H} S_V(t)\cfrac{A_H(t)+I_H(t)}{P_H(t)}-\mu  S_V(t),\\[1.55mm]
E'_V(t)&= b q_{V\leftarrow H} S_V(t)\cfrac{A_H(t)+I_H(t)}{P_H(t)} - (\nu_V +\mu) E_H,\\[1.5mm]
I_V'(t)&= \nu_V I_V(t)-\mu I_V(t),
\end{array} 
\right.
\end{equation}
where $b$ is the mosquito biting rate, $q_{H\leftarrow V}$ and $q_{V\leftarrow H}$ are the probabilities of transmission per bite from vector to humans and from humans to vector, respectively, $1/\nu_H$ and $1/\nu_V$ are the average latent periods for humans and mosquitoes, respectively,  $1/\gamma$ is the average human infectious period, $1/\mu $ is the mosquito average life span, and $\Lambda(t)\ge 0$ is the $\tau$-periodic mosquito inflow rate.
Note that $P_H'=0$, thus $P_H$ is constant, while $P_V$ satisfies the linear ODE $P_V'(t)=\Lambda(t)-\mu P_V(t).$ We take $\Lambda(t)=\left(1+\varepsilon\cos\left(\frac{2\pi t}{\tau}\right)\right)\mu P_V(t)$, for $\varepsilon\in (0,1)$, so that, for every $P_V(0)\ge0$, 
system \eqref{VBHsystemH}-\eqref{VBVsystemH} admits the unique DFPS $S_H^*(t)\equiv P_H,\ E_H^*(t)=A_H^*(t)=I_H^*(t)=E_V^*(t)=I_V^*(t)\equiv 0,\ S_V^*(t)=P_V(t)$, where 
$P_V(t)=P_V(0)f(t)$ with $f(t)=\exp\left(\frac{\varepsilon\tau \mu }{2\pi}\sin\left(\frac{2\pi t}{\tau}\right)\right)$.
By defining 
the infection matrices $[B_H]_{j,k}=\delta_{(j,k), (1,5)}bq_{H\leftarrow V}$, $[B_V(t)]_{j,k}=\left[\delta_{(j,k), (4,2)} + \delta_{(j,k), (4,3)}\right]bq_{V\leftarrow H} r f(t)$,
$j,k=1,\dots, 5$ for $\delta$ the Dirac delta and $r :=P_V(0)/P_H$, and the baseline transition matrix 
\begin{equation*}
M_0=
\begin{pmatrix}
 \nu_H & 0 & 0 & 0 & 0  \\
 -(1-\omega)\nu_H & \gamma & 0 & 0 & 0  \\
 -\omega\nu_H & 0 & \gamma & 0 & 0  \\
 0 & 0 & 0 & \nu_V+\mu  & 0  \\
 0 & 0 & 0 & -\nu_V & \mu  
\end{pmatrix},
\end{equation*} 
the linearized equations for the infected individuals around the DFPS 
can be recast as in \eqref{ODEBM} for each of the following choices:
\begin{enumerate}
\item \label{R0} $B=B_H+B_V$ and $M\equiv M_0$,
\item \label{TH} $B\equiv B_H$ and $M=M_0-B_V$,
\item \label{TV} $B=B_V$ and $M\equiv M_0-B_H$.
\end{enumerate}
Each of these corresponds to a different reproduction number: the BRN $R_0$ under \ref{R0} (as $B$ contains all birth terms); the TRN for the host population $T_H$  under \ref{TH}; and the TRN for the vector population $T_V$ under \ref{TV}.
Indeed, using notation similar to those previously introduced, it is not difficult to see that, e.g.,  under \ref{TH}, 
$\mathcal M$ and the positive operator $\mathcal K_V:=\mathcal B_V\mathcal M_0^{-1}$ satisfy $s(-\mathcal M)<0$ and $\rho(\mathcal K_V)<1$, respectively. Hence, by letting $\mathcal K_H:=\mathcal B_H\mathcal M_0^{-1}$, the operator
$\mathcal K_H\sum_{m=0}^\infty(\mathcal K_V)^m=\mathcal K_H(I_X-\mathcal K_V)^{-1}=\mathcal B_H(\mathcal M_0-\mathcal B_V)^{-1}=:\mathcal B\mathcal M^{-1}$
is well-defined and positive, and 
$T_H:=\rho(\mathcal B\mathcal M^{-1})$ \cite{bacaer2012model, Inaba2012}. 
We take parameters inspired by the Chikungunya outbreak in Italy of 2007 and adapted from \cite{guzzetta2016, poletti2011} (see also \cite{feng2019modelling} and references therein): $\tau =1$ (years), $b=0.09$ (days$^{-1}$), $q_{H\leftarrow V}=65\%$, $q_{V\leftarrow H}=85\%$, $1/\nu_H=2.5$ (days), $1/\nu_V=1.5$ (days), $\omega=82\%$, $1/\gamma= 7$ (days) and $1/\mu =20$ (days). Finally, we assume  $r=10$ and, following \cite{heesterbeek1995threshold}, 
 $\delta\coloneqq \frac{\varepsilon\tau\mu }{2\pi} <1$, 
so that $R_0^2\in [2.8, 3.4]$, in agreement with \cite{poletti2011} (where the authors provide an estimate $R_0^2\approx 3.3$). For the readers' convenience, parameters are listed in \Cref{TableSIR2}. \begin{table}
\footnotesize\caption{Parameters used for the numerical experiments of \cref{SectionType} and adapted from \cite{guzzetta2016, poletti2011}.}
\label{TableSIR2}
\begin{center}
\begin{tabular}{c  c c }
\rowcolor{gray!20}
Parameter & interpretation & value\\[0.5mm]
\hline
$\tau$ & time period & 1 (years)\\[0.5mm]
$b$ & mosquito biting rate  & 0.09 (days)\\[0.5mm]
$q_{H\leftarrow V}$ & probability of transmission per bite from mosquito to human & 65\% \\[0.5mm]
$q_{V\leftarrow H}$ & probability of transmission per bite from human to mosquito & 85\% \\[0.5mm]
$1/\nu_H$ &  average human latent period  & 2.5 (days)\\[0.5mm]
$1/\nu_V$ &  average mosquito latent period
 & 1.5 (days)\\[0.5mm]
$\omega$ & probability of developing symptoms for humans & 82\%\\[0.5mm]
$1/\gamma$ &  average human infectious period & 3.5 (days) \\[0.5mm]
$1/\mu$ & average mosquito life span & 20 
(days)\\[0.5mm]
$r$ & ratio between initial mosquito and human populations & 10 \\
\hline
\end{tabular}
\end{center}
\end{table}

In the following, in light of the results of \cref{SectionSIR}, we only use the Fourier method, as the model coefficients are real analytic.
\begin{figure}
\begin{center}
\includegraphics[width=1.\textwidth]{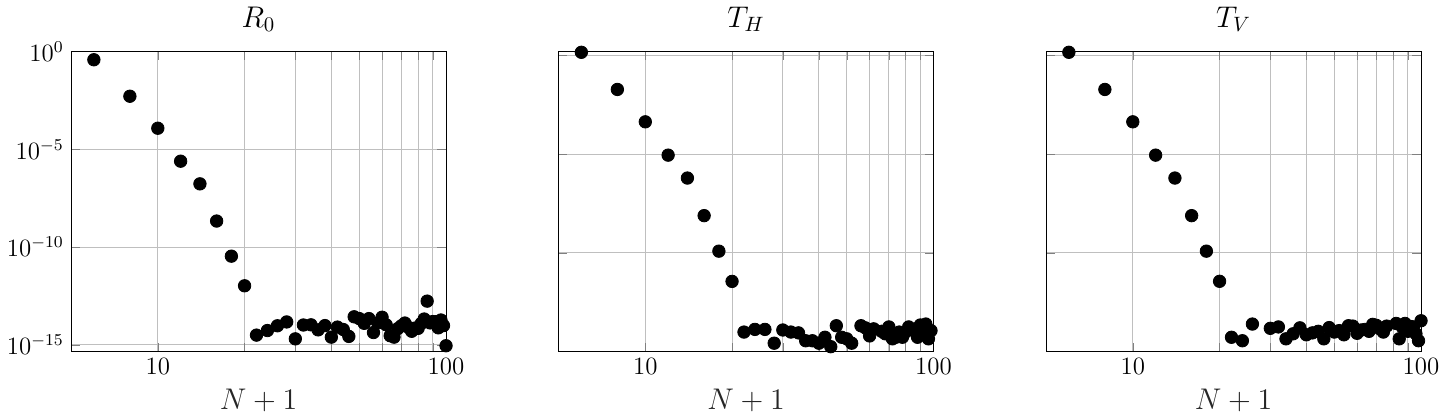}
\caption{\Cref{SectionType}: log-log plot of the error on $R_0$  (left), $T_H$ (center) and $T_V$ (right) with the Fourier method for increasing (odd) $N$, for $\delta=0.8$. Reference values $R_0\approx 1.785206757378205$,  $T_H\approx 3.186963166588779$ and $T_V\approx 3.186963166588749$ computed with $N=151$.}\label{SEAIR_SEIconv}
\end{center}
\end{figure}
\Cref{SEAIR_SEIconv} shows infinite order of convergence for all reproduction numbers, in accordance with \cref{convanal}.
\begin{table}
\footnotesize
\caption{\Cref{SectionType}: computed values for $R_0$, $R_0^2$, $T_H$, $T_V$, and relevant computational times using the Fourier method with $N=25$, for $\delta=0.8$.}
\label{SAEIR_SI_table}
\begin{center}
\begin{tabular}{c c c}
\rowcolor{gray!20}
Reproduction number & value &  CPU time (s) \\
[0.5mm]
\hline
$R_0$ & 1.785206757378213
 &  $8.0\times 10^{-3}$ \\
$R_0^2$ &    3.186963166588834
 &   \\
$T_H$ &  3.186963166588752 &  $6.5\times 10^{-3}$ \\
$T_V$ &    3.186963166588752
 &  $5.6\times 10^{-3}$ \\
\hline
\end{tabular}
\end{center}
\end{table}
\Cref{SAEIR_SI_table} shows the computed values of $R_0, T_H, T_V$, and relevant CPU times, for $N=25$. Note the well-known relation between $R_0$ and the TRNs for a one-vector-one-host model, i.e. $R_0^2=T_H=T_V$ \cite{bacaer2012model, heesterbeek1995threshold2, HEESTERBEEK20073}.
\begin{figure}
\begin{center}
\includegraphics[width=1.\textwidth]{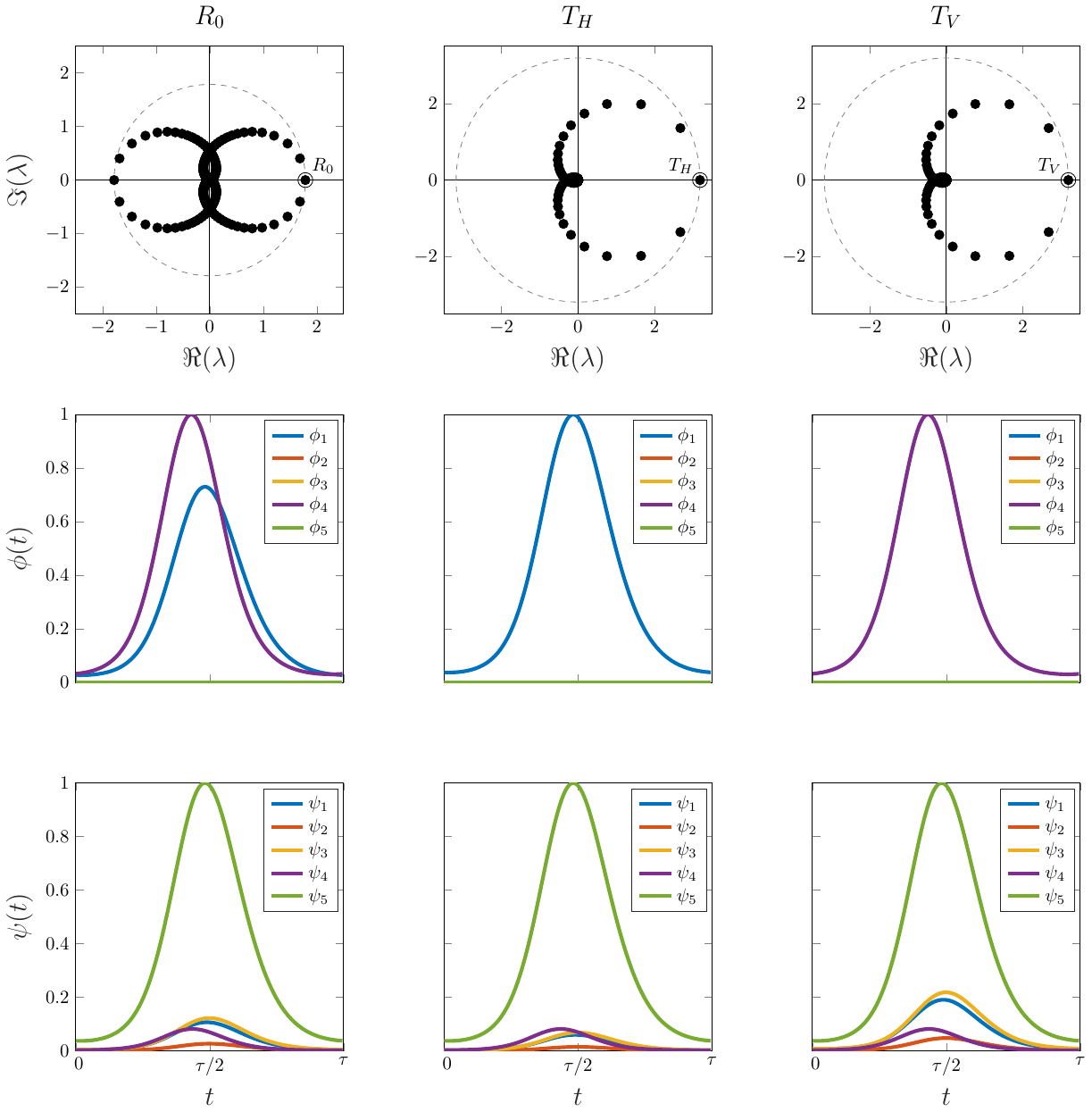}
\caption{\Cref{SectionType}: computed spectra (upper row), eigenfunction (center row) and generalized eigenfunction (lower row) relevant to the reproduction number under \ref{R0} (left column), \ref{TH} (center column) and \ref{TV} (right column), for the Fourier method with $N=201$ and $\delta =0.8$. Note that, in the center row, the only nonzero components of $\phi$ are those relevant to the exposed classes, as all newly infected individuals are exposed. In particular, in the left panel, both $\phi_1$ and $\phi_4$ are nonzero, while in the center and right panels, the only nonzero component is $\phi_1$ and $\phi_4$, respectively. See for comparison the plot concerning the generalized eigenfunctions.}\label{SEAIR_SEI_eig}
\end{center}
\end{figure}
In \Cref{SEAIR_SEI_eig}, we plot spectra and eigenfunctions relevant to $R_0$, $T_H$ and $T_V$ under \ref{R0}, \ref{TH} and \ref{TV}, respectively. Note the different structures of the spectra between the case \ref{R0} and \ref{TH}-\ref{TV}.
\begin{figure}
\begin{center}
\includegraphics[width=1.\textwidth]{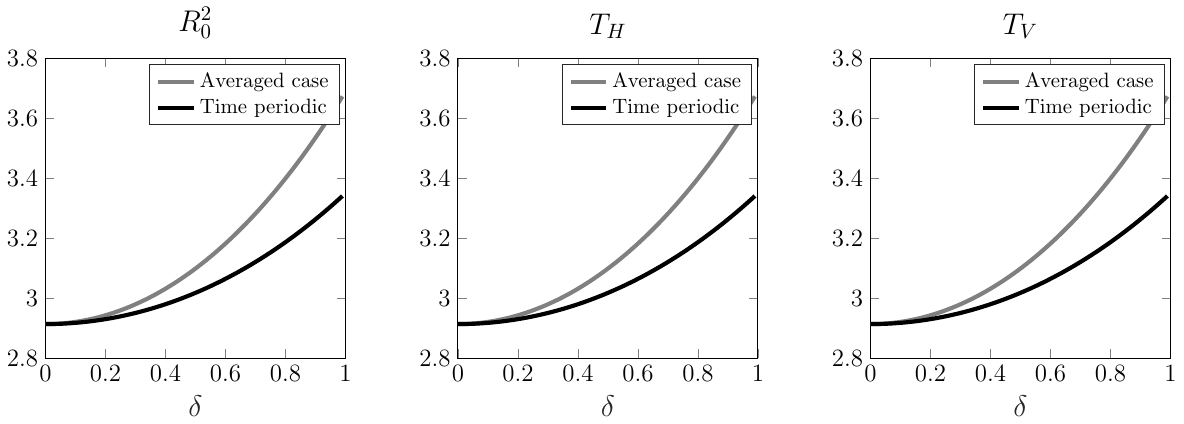}
\caption{\Cref{SectionType}: $R_0^2$ (left), $T_H$ (center) and $T_V$ (right) as functions of $\delta=\varepsilon\tau\mu/(2\pi)$ computed with the Fourier method using $N=25$. Note that the reproduction numbers for the averaged model also vary with $\delta$, as the average of $f$ over $[0, \tau]$ differs from 1.}\label{SEAIR_SEI_bif}
\end{center}
\end{figure}
Finally, \Cref{SEAIR_SEI_bif} shows the behavior of $R_0, T_H$ and $T_V$ for increasing values of $\delta \in [0, 1)$. For comparison, we include the same results for the averaged model. Note that for this choice of the parameters, the reproduction number for the time-periodic problem is always lower than that for the averaged one. In particular, the discrepancy between the two values increases with $\delta$, in accordance with \cite{Bacaer2007}.

\section*{Acknowledgements}
The authors are grateful to Rossana Vermiglio for many useful discussions. 
DB and SDR are members of the INdAM Research group GNCS and of the UMI Research group
“Modellistica socio-epidemiologica”. The work of DB and SDR was partially supported by the Italian Ministry of University and Research (MUR) through the PRIN 2020 project (No. 2020JLWP23) “Integrated Mathematical Approaches to Socio-Epidemiological Dynamics”, Unit of Udine (CUP: G25F22000430006). SDR is supported by the project “One Health Basic and Translational Actions Addressing
Unmet Needs on Emerging Infectious Diseases” (INF-ACT), BaC “Behaviour and sentiment monitoring and modelling for outbreak control/BEHAVE-MOD” (PE00000007 – CUP I83C22001810007)
funded by the NextGenerationEU. JR is part of the Catalan research group 2021 SGR 00113. JR is supported by grants PID2021- 123733NB-I00 and RED2022-134784-T funded by the Spanish Ministerio de Ciencia, Inovación y Universidades.
This research was partially conducted while SDR was visiting JR at the Department of Computer Science, Applied Mathematics and Statistics of the University of Girona, Spain, and while SDR was a PhD student at the Department of Mathematics, Computer Science and Physics of the University of Udine.

\printbibliography
\end{document}